\RequirePackage{fix-cm}
\documentclass[french,english]{article}
\usepackage[T1]{fontenc}
\usepackage[latin9]{inputenc}
\usepackage[a4paper]{geometry}
\usepackage{fancyhdr}
\usepackage{color}
\usepackage{babel}
\addto\extrasfrench{%
   \providecommand{\fg}{\ifdim\lastskip>\z@\unskip\fi~\frqq}%
}

\usepackage{float}
\usepackage{amsthm}
\usepackage{amsmath}
\usepackage{amssymb}
\usepackage{fixltx2e}
\usepackage{graphicx}
\usepackage{esint}
\usepackage[unicode=true,pdfusetitle,
 bookmarks=true,bookmarksnumbered=true,bookmarksopen=true,bookmarksopenlevel=1,
 breaklinks=true,pdfborder={0 0 1},backref=false,colorlinks=true]
 {hyperref}
\hypersetup{
 linkcolor=black, citecolor=blue, urlcolor=blue, filecolor=blue,pdfpagelayout=OneColumn, pdfnewwindow=true,pdfstartview=XYZ, plainpages=false, pdfpagelabels, hyperindex=true}
\usepackage{breakurl}

\makeatletter

\providecommand{\tabularnewline}{\\}

\usepackage{enumitem}		
\numberwithin{equation}{section}
\numberwithin{figure}{section}
\numberwithin{table}{section}
  \theoremstyle{plain}
  \newtheorem{thm}{\protect\theoremname}[section]
  \theoremstyle{plain}
  \newtheorem{lem}{\protect\lemmaname}[section]
  \theoremstyle{plain}
  \newtheorem{prop}{\protect\propositionname}[section]

\usepackage{fancyhdr}
\usepackage{lastpage}
\usepackage{amsmath}
\usepackage{amssymb} 
\usepackage{graphicx}
\usepackage{color}
\usepackage{fancybox} 
\usepackage{moreverb} 
\usepackage{hangcaption} 
\usepackage{listings} 

\def\esssup_#1{\underset{#1}{\mathrm{ess\,sup\, }}}

\usepackage{ifpdf} 
\ifpdf 

 \IfFileExists{lmodern.sty}{\usepackage{lmodern}}{}

\fi 

\let\myTOC\tableofcontents
\renewcommand\tableofcontents{%
  \frontmatter
  \pdfbookmark[1]{\contentsname}{}
  \myTOC
  \mainmatter }

\def\LyX{\texorpdfstring{%
  L\kern-.1667em\lower.25em\hbox{Y}\kern-.125emX\@}
  {LyX}}

\@addtoreset{footnote}{section}

\usepackage{eurosym}

\usepackage{enumitem}
\setlist{leftmargin=*, topsep=0.5em, parsep=0pt, itemsep=1em, labelindent=0pt, align=left}

\@ifundefined{showcaptionsetup}{}{%
 \PassOptionsToPackage{caption=false}{subfig}}
\usepackage{subfig}
\makeatother

  \addto\captionsenglish{\renewcommand{\lemmaname}{Lemma}}
  \addto\captionsenglish{\renewcommand{\propositionname}{Proposition}}
  \addto\captionsfrench{\renewcommand{\lemmaname}{Lemme}}
  \addto\captionsfrench{\renewcommand{\propositionname}{Proposition}}
  \providecommand{\lemmaname}{Lemma}
  \providecommand{\propositionname}{Proposition}
\addto\captionsenglish{\renewcommand{\theoremname}{Theorem}}
\addto\captionsfrench{\renewcommand{\theoremname}{Th�or�me}}
\providecommand{\theoremname}{Theorem}

\begin{document}

\title{A numerical algorithm for fully nonlinear HJB equations:\\
an approach by control randomization}


\author{Idris Kharroubi\footnote{The research of the author benefited from the support of the French ANR research grant LIQUIRISK (ANR-11-JS01-0007).}\\ \footnotesize{CEREMADE, CNRS UMR 7534}, \\\footnotesize{Universit\'e Paris Dauphine}\\ \footnotesize{and CREST,} \\
\footnotesize{ \texttt{kharroubi at  ceremade.dauphine.fr}} 
  \and Nicolas Langren\'e\\ \footnotesize{Laboratoire de Probabilit\'es et Mod\`eles Al\'eatoires,}\\ \footnotesize{Universit\'e Paris Diderot}\\ \footnotesize{ and EDF R\&D}\\\footnotesize{\texttt{langrene at math.univ-paris-diderot.fr}}
   \and Huy\^en Pham\\ \footnotesize{Laboratoire de Probabilit\'es et Mod\`eles Al\'eatoires,}\\ \footnotesize{Universit\'e Paris Diderot}\\ \footnotesize{and CREST-ENSAE}\\\footnotesize{\texttt{pham  at math.univ-paris-diderot.fr}}
}

\maketitle
\begin{abstract}
We propose a probabilistic numerical algorithm  to solve Backward Stochastic
Differential Equations (BSDEs) with nonnegative jumps, a class of BSDEs introduced in \cite{Kharroubi12} for representing fully nonlinear
HJB equations. In particular, this allows us  to numerically solve stochastic
control problems with controlled volatility, possibly degenerate.  Our backward scheme, based on least-squares regressions, takes advantage of high-dimensional properties of Monte-Carlo methods, and also provides a parametric estimate in feedback form for the optimal control. A partial analysis of the error of the scheme is provided, as well as numerical tests on the
problem of superreplication of option with uncertain volatilities and/or correlations, including a detailed comparison with the numerical
results from the alternative scheme proposed in \cite{Guyon11}.
\end{abstract}

\vspace{5mm}


\vspace{5mm}


\newpage

\section{Introduction}

Consider the following general Hamilton-Jacobi-Bellman (HJB) equation:

\begin{align}
\frac{\partial v}{\partial t}+\sup_{a\in A}\left\{ b\left(x,a\right).D_{x}v+\frac{1}{2}\mathrm{tr}\left(\sigma\sigma^{\top}\left(x,a)\right)D_{x}^{2}v\right)+f\left(x,a,v,\sigma^{\top}\left(x,a\right).D_{x}v\right)\right\}  & =0\,\,,\,\left(t,x\right)\in\left[0,T\right)\times\mathbb{R}^{d}\label{eq:HJB}\\
v\left(T,x\right) & =g\left(x\right)\,,\, x\in\mathbb{R}^{d}\nonumber 
\end{align}

where $A$ is a bounded subset of $\mathbb{R}^{q}$. It is well known
that the HJB equation \eqref{eq:HJB} is the dynamic programming equation
for the following stochastic control problem:
\begin{align}
v\left(t,x\right) & =\sup_{\alpha\in\mathcal{A}}\mathbb{E}^{t,x}\left[\int_{t}^{T}f\left(X_{s}^{\alpha},\alpha_{s}\right)ds+g\left(X_{T}^{\alpha}\right)\right]\label{eq:ControlPb}\\
dX_{s}^{\alpha} & =b\left(X_{s}^{\alpha},\alpha_{s}\right)ds+\sigma\left(X_{s}^{\alpha},\alpha_{s}\right)dW_{s}\nonumber 
\end{align}

Moreover, it is proved in \cite{Kharroubi12} that this HJB equation
admits a probabilistic representation by means of a BSDE with nonpositive
jumps. We recall below this construction.

Introduce a Poisson random measure $\mu_{A}\left(dt,da\right)$ on
$\mathbb{R}_{+}\times A$ with finite intensity measure $\lambda_{A}\left(da\right)dt$
associated to the marked point process $\left(\tau_{i},\zeta_{i}\right)_{i}$,
independent of $W$, and consider the pure jump process $\left(I_{t}\right)_{t}$,
valued in $A$, defined as follows:
\[
I_{t}=\zeta_{i}\,,\,\,\tau_{i}\leq t<\tau_{i+1}\, ,
\]
and interpreted as a randomization of the control process $\alpha$. 

Next, consider the uncontrolled forward regime switching diffusion
process
\[
dX_{s}=b\left(X_{s},I_{s}\right)ds+\sigma\left(X_{s},I_{s}\right)dW_{s}\,.
\]
Observe that the pair process $\left(X,I\right)$ is Markov. Now,
consider the following BSDE with jumps w.r.t. the Brownian-Poisson
filtration $\mathbb{F}=\mathbb{F}^{W,\mu_{A}}=\left(\mathcal{F}_{t}\right)_{0\leq t\leq T}$.

\begin{equation}
Y_{t}=g\left(X_{T}\right)+\int_{t}^{T}f\left(X_{s},I_{s},Y_{s},Z_{s}\right)ds-\int_{t}^{T}Z_{s}dW_{s}-\int_{t}^{T}\int_{A}U_{s}\left(a\right)\tilde{\mu}_{A}\left(ds,da\right)\label{eq:unconstrainedBSDE}
\end{equation}
where $\tilde{\mu}_{A}$ is the compensated measure of $\mu_{A}$.

Finally, we constrain the jump component of the BSDE \eqref{eq:unconstrainedBSDE}
to be nonpositive, i.e.
\[
U_{t}\left(a\right)\leq0,\,\, d\mathbb{P}\otimes dt\otimes\lambda\left(da\right)\, a.e.
\]

We denote by $\bar A$ $>$ $0$ an upper bound for the compact set $A$ of $\mathbb{R}^q$, i.e. $|a|$ $\leq$ $\bar A$ for all $a$ $\in$ $A$, and we make the standing assumptions:
\begin{enumerate}
\item The functions $b$ and $\sigma$ are Lipschitz: there exists  $L_{b,\sigma}$ $>$ $0$ s.t.
\begin{eqnarray*}
|b(x_1,a_1) - b(x_2,a_2)| +  |\sigma(x_1,a_1) - \sigma(x_2,a_2)|  & \leq & L_{b,\sigma}\big(|x_1-x_2| + |a_1-a_2| \big), 
\end{eqnarray*}
for all $x_1,x_2$ $\in$ $\mathbb{R}^d$, $a_1,a_2$ $\in$ $A$. 
\item The functions $f$ and $g$ are Lipschitz continuous: there exists
$L_{g},L_{f}>0$ s.t. 
\begin{eqnarray*}
\left|g\left(x_{1}\right)-g\left(x_{2}\right)\right| & \leq & L_{g}\left|x_{1}-x_{2}\right|\\
\left|f\left(x_{1},a_{1},y_{1},z_{1}\right)-f\left(x_{2},a_{2},y_{2},z_{2}\right)\right| & \leq & L_{f}\left(\left|x_{1}-x_{2}\right|+\left|a_{1}-a_{2}\right|+\left|y_{1}-y_{2}\right|+\left|z_{1}-z_{2}\right|\right),
\end{eqnarray*}
for all $x_1,x_2$ $\in$ $\mathbb{R}^d$, $a_1,a_2$ $\in$ $A$. 
\end{enumerate}
Under these conditions, we consider the minimal solution $\left(Y,Z,U,K\right)$
of the following constrained BSDE:
\begin{eqnarray}
Y_{t} & = & g\left(X_{T}\right)+\int_{t}^{T}f\left(X_{s},I_{s},Y_{s},Z_{s}\right)ds-\int_{t}^{T}Z_{s}dW_{s}\label{eq:constrained-BSDE}\\
 &  & +K_{T}-K_{t}-\int_{t}^{T}\int_{A}U_{s}\left(a\right)\tilde{\mu}_{A}\left(ds,da\right)\,,\,\,0\leq t\leq T\,,\, a.s. \nonumber 
\end{eqnarray}
subject to the constraint
\begin{equation}
U_{t}\left(a\right)\leq0\,,\,\, d\mathbb{P}\otimes dt\otimes\lambda\left(da\right)\, a.e.\,\mathrm{on}\,\Omega\times\left[0,T\right]\times A\label{eq:constraint}
\end{equation}

By the Markov property of $\left(X_{t},I_{t}\right)$, there exists
a deterministic function $y=y\left(t,x,a\right)$ such that the minimal
solution to \eqref{eq:constrained-BSDE}-\eqref{eq:constraint} satisfies $Y_{t}=y\left(t,X_{t},I_{t}\right)$,
$0\leq t\leq T$.
\begin{thm}
\label{Th:HJB}\cite{Kharroubi12} $y=y\left(t,x,a\right)$ does not
depend on $a$: $y=y\left(t,x\right)$, and is a viscosity solution
of the HJB equation \eqref{eq:HJB}:
\begin{align*}
\frac{\partial y}{\partial t}+\sup_{a\in A}\left\{ b\left(x,a\right).D_{x}y\left(t,x\right)+\frac{1}{2}\mathrm{tr}\left(\sigma\sigma^{\top}\!\!\left(x,a\right)D_{x}^{2}v\left(t,x\right)\right)+f\left(x,a,y,\sigma^{\top}\!\!\left(x,a\right)D_{x}y\right)\right\}  & =0\\
\left(t,x\right)\in\left[0,T\right)\times\mathbb{R}^{d}\hspace{90mm}v\left(T,x\right) & =g\left(x\right),x\!\in\!\mathbb{R}^{d}
\end{align*}

\end{thm}
Now, the aim of this paper is to provide a numerical scheme for computing
an approximation of the solution of the constrained BSDE \eqref{eq:constrained-BSDE}-\eqref{eq:constraint}.
In light of Theorem \ref{Th:HJB}, this will provide an approximation
of the solution of the general HJB equation \eqref{eq:HJB}, which
encompasses  stochastic control problems such  as the one described
in equation \eqref{eq:ControlPb}, ie. problems where both the drift
and the volatility of the underlying diffusion can be controlled, including degenerate diffusion coefficient.

The outline of the subsequent sections is the following.

First, Section \ref{sec:Scheme} describes our scheme. We start from a time-discretization of the problem,  proposed in \cite{Kharroubi13}, which 
gives rise to a  backward scheme involving the simulation of the forward regime switching process $(X,I)$, hence taking advantage of high-dimensional 
properties of Monte-Carlo methods. The final step towards an implementable scheme is to approximate the conditional
expectations that arise from this scheme. Here we  use  empirical
least-squares regression, as this method provides a parametric estimate
in feedback form of the optimal control. A partial analysis of the
impact of this approximation is provided, and the remaining obstacles
towards a full analysis are highlighted.

Then, Section \ref{sec:Applications} is devoted to numerical tests
of the scheme on various examples. The major application that takes
advantage of the possibilities of our scheme is the problem of pricing
and hedging contingent claims under uncertain volatility and (for
multi-dimensional claims) correlation. Therefore most of this section
is devoted to this specific application. To our knowledge, the only
other Monte Carlo scheme for HJB equations that can handle continuous
controls as well as controlled volatility is described in \cite{Guyon11},
where they make use of another generalization of BSDEs, namely second-order
BSDEs. Therefore we compare the performance of our scheme to the results
provided in their paper.

Finally, Section \ref{sec:Conclusion} concludes the paper.

\section{Regression scheme\label{sec:Scheme}}

Define a deterministic time grid $\text{\ensuremath{\pi}:=\ensuremath{\left\{  0=t_{0}<\ldots<t_{N}=T\right\} } }$
for the interval $[0,T]$, with mesh $\left|\pi\right|:=\max_{0\leq i<N}\Delta_{i}$
where $\Delta_{i}:=t_{i+1}-t_{i}$. Denote by $\mathbb{E}_{i}\left[.\right]:=\mathbb{E}\left[.\left|\mathcal{F}_{t_{i}}\right.\right]=\mathbb{E}\left[.\left|X_{i},I_{i}\right.\right]$.
The discretization of the constrained BSDE \eqref{eq:constrained-BSDE}-\eqref{eq:constraint} 
can be written as follows:

\begin{equation}
\begin{cases}
Y_{N} & =g\left(X_{N}\right)\\
\Delta_{i}\mathcal{Z}_{i} & =\mathbb{E}_{i}\left[Y_{i+1}\Delta W_{i}^{\top}\right]\\
\mathcal{Y}_{i} & =\mathbb{E}_{i}\left[Y_{i+1}+f\left(X_{i},I_{i},Y_{i+1},\mathcal{Z}_{i}\right)\Delta_{i}\right]\\
Y_{i} & =\esssup_{a\in A}\mathbb{E}_{i,a}\left[\mathcal{Y}_{i}\right]
\end{cases}\label{eq:discreteBSDE}
\end{equation}

where $\mathbb{E}_{i,a}\left[.\right]:=\mathbb{E}\left[.\left|X_{i},I_{i},I_{i}=a\right.\right]=\mathbb{E}\left[.\left|X_{i},I_{i}=a\right.\right]$.

First, remark that, from the Markov property of $\left(X_{i},I_{i}\right)_{1\leq i\leq N}$,
there exist deterministic functions $\tilde{y}_{i}$ and $\tilde{z}_{i}$
such that $\left(\mathcal{Y}_{i},\mathcal{Z}_{i}\right)=\left(\tilde{y}_{i}\left(X_{i},I_{i}\right),\tilde{z}_{i}\left(X_{i},I_{i}\right)\right)$.
Hence $\mathcal{Y}_{i}$ and $\mathcal{Z}_{i}$ can be seen as intermediate
quantities towards the discrete-time approximation of the BSDE \eqref{eq:constrained-BSDE}-\eqref{eq:constraint} 
$\left(Y_{i},Z_{i}\right)=\left(y_{i}\left(X_{i}\right),z_{i}\left(X_{i}\right)\right)$
, which do not depend on $I_{i}$.

Formally, the jump constraint \eqref{eq:constraint} states that $\tilde{y}_{i}\left(X_{i},a\right)-y_{i}\left(X_{i}\right)=U_{t}\left(a\right)\leq0$
a.s., meaning that the minimal solution satisfies $Y_{i}=y_{i}\left(X_{i}\right)=\esssup_{a\in A}\tilde{y}_{i}\left(X_{i},a\right)=\esssup_{a\in A}\mathbb{E}_{i,a}\left[\mathcal{Y}_{i}\right]$. 

Moreover, one can extract $Z_{i}$ from the scheme if needed. Indeed,
denoting $a^{*}=\arg\esssup_{a\in A}\mathbb{E}_{i,a}\left[\mathcal{Y}_{i}\right]$,
i.e. $Y_{i}=\mathbb{E}_{i,a^{*}}\left[\mathcal{Y}_{i}\right]$, then
$Z_{i}=z_{i}\left(X_{i}\right)=\tilde{z}_{i}\left(X_{i},a^{*}\right)$. 

Finally remark that the numerical scheme \eqref{eq:discreteBSDE}
is explicit, as we choose to define $\mathcal{Y}_{i}$ as a function
of $Y_{i+1}$ and not of $Y_{i}$.

The convergence of the solution of the discretized scheme \eqref{eq:discreteBSDE}
towards the solution of the constrained BSDE \eqref{eq:constrained-BSDE}-\eqref{eq:constraint}  
is thoroughly examined in \cite{Kharroubi13}. In this paper, we start
from the discrete version \eqref{eq:discreteBSDE} and derive an implementable
scheme from it.

Indeed, the discrete scheme \eqref{eq:discreteBSDE} is in general
not readily implementable because it involves conditional expectations
that cannot be computed explicitly. It is thus necessary in practice
to approximate these conditional expectations. Here we follow the
empirical regression approach (\cite{Lemor06,Bouchard11,Gobet11,Zanger12,Aid12-2}).
In our context, apart from being easy to implement, the strong advantage
of this choice is that, unlike other standard methods, it provides
as a by-product a parametric feedback estimate $\hat{\alpha}\left(t,X_{t}\right)$
for the optimal control. 

The idea is to replace the conditional expectations from \eqref{eq:discreteBSDE}
by empirical regressions. This section is devoted to the analysis
of the error generated by this replacement.

\subsection{Localizations}

The first step is to localize the discrete BSDE \eqref{eq:discreteBSDE},
i.e. to truncate it so that it admits a.s. deterministic bounds. Introduce
$R_{X}\in\mathbb{R}_{+}^{d}$ and $R_{w}\in\mathbb{R}_{+}$ and define
the following truncations of $X_{i}$ and $\Delta W_{i}$:
\begin{align}
\left[X_{i}\right]_{X} & :=-R_{X}\vee X_{i}\wedge R_{X}=\left\{ -R_{1,X}\vee X_{1,i}\wedge R_{1,X},\ldots,-R_{d,X}\vee X_{d,i}\wedge R_{d,X}\right\} ^{\top}\label{eq:Xtruncation}\\
\left[\Delta W_{i}\right]_{w} & :=-R_{w}\sqrt{\Delta_{i}}\vee\Delta W_{i}\wedge R_{w}\sqrt{\Delta_{i}}=\left\{ -R_{w}\sqrt{\Delta_{i}}\vee\Delta W_{1,i}\wedge R_{w}\sqrt{\Delta_{i}},\ldots,-R_{w}\sqrt{\Delta_{i}}\vee\Delta W_{q,i}\wedge R_{w}\sqrt{\Delta_{i}}\right\} ^{\top}\label{eq:Wtruncation}
\end{align}
 Define $R=\left\{ R_{X},R_{w}\right\} $ and define the localized
version of the discrete BSDE \eqref{eq:discreteBSDE}, using the truncations
\eqref{eq:Xtruncation} and \eqref{eq:Wtruncation}.
\begin{equation}
\begin{cases}
Y_{N}^{R} & =g\left(\left[X_{N}\right]_{X}\right)\\
\Delta_{i}\mathcal{Z}_{i}^{R} & =\mathbb{E}_{i}\left[Y_{i+1}^{R}\left[\Delta W_{i}^{\top}\right]_{w}\right]\\
\mathcal{Y}_{i}^{R} & =\mathbb{E}_{i}\left[Y_{i+1}^{R}+f\left(\left[X_{i}\right]_{X},I_{i},Y_{i+1}^{R},\mathcal{Z}_{i}^{R}\right)\Delta_{i}\right]\\
Y_{i}^{R} & =\esssup_{a\in A}\mathbb{E}_{i,a}\left[\mathcal{Y}_{i}^{R}\right]
\end{cases}\label{eq:localizedBSDE}
\end{equation}

First, we check that this localized BSDE does admit a.s. bounds.
\begin{lem}
\label{lem:asBounds}{[}almost sure bounds{]} For every $R=\left\{ R_{X},R_{w}\right\} \in\left[0,\infty\right)^{d}\times\left[0,\infty\right]$
and every $1\leq i\leq N$, the following uniform bounds hold a.s.:
\begin{eqnarray*}
\left|\mathcal{Y}_{i}^{R}\right|,\left|Y_{i}^{R}\right| & \leq & C_{y}=C_{y}\left(R_{X}\right):=e^{\frac{C}{2}T}\sqrt{C_{g}^{2}\left(R_{X}\right)+\frac{e^{C\left|\pi\right|}}{L_{f}^{2}}C_{f}^{2}\left(R_{X}\right)}\\
\left|\mathcal{Z}_{i}^{R}\right|,\left|Z_{i}^{R}\right| & \leq & C_{z}=C_{y}\left(R_{X}\right):=\frac{\sqrt{q}}{\sqrt{\Delta_{i}}}C_{y}
\end{eqnarray*}

where $C:=3L_{f}^{2}\left(q+\left|\pi\right|\right)+\frac{1}{q}$,
\textup{$C_{g}\left(R_{X}\right):=\max_{-R_{X}\leq x\leq R_{X}}\left|g\left(x\right)\right|$
and $C_{f}\left(R_{X}\right):=L_{f}\left(\left|R_{X}\right|+\left|\bar{A}\right|\right)+f\left(0,0,0,0\right)$}\end{lem}
\begin{proof}
First, as $g$ is continuous, there exists $C_{g}=C_{g}\left(R_{X}\right)>0$
such that for all $-R_{X}\leq x\leq R_{X}$, $\left|g\left(x\right)\right|\leq C_{g}\left(R_{X}\right)$.
Hence 
\begin{equation}
\left(Y_{N}^{R}\right)^{2}=g^{2}\left(\left[X_{N}\right]_{X}\right)\leq C_{g}^{2}\left(R_{X}\right)\label{eq:asboundFinal}
\end{equation}
Next, 
\[
\Delta_{i}\mathcal{Z}_{i}^{R}=\mathbb{E}_{i}\left[Y_{i+1}^{R}\left[\Delta W_{i}\right]_{w}\right]=\mathbb{E}_{i}\left[\left(Y_{i+1}^{R}-\mathbb{E}_{i}\left[Y_{i+1}^{R}\right]\right)\left[\Delta W_{i}\right]_{w}\right]
\]
Thus, using the Cauchy-Schwarz inequality and dividing by $\Delta_{i}$:
\[
\Delta_{i}\left(\mathcal{Z}_{i}^{R}\right)^{2}\leq q\left(\mathbb{E}_{i}\left[\left(Y_{i+1}^{R}\right)^{2}\right]-\mathbb{E}_{i}\left[Y_{i+1}^{R}\right]^{2}\right)
\]
Now, using Young's inequality $\left(a+b\right)^{2}\leq\left(1+\gamma\Delta_{i}\right)a^{2}+\left(1+\frac{1}{\gamma\Delta_{i}}\right)b^{2}$
with $\gamma>0$:

\[
\left(\mathcal{Y}_{i}^{R}\right)^{2}\leq\left(1+\gamma\Delta_{i}\right)\mathbb{E}_{i}\left[Y_{i+1}^{R}\right]^{2}+\left(1+\frac{1}{\gamma\Delta_{i}}\right)\Delta_{i}^{2}\mathbb{E}_{i}\left[f^{2}\left(\left[X_{i}\right]_{X},I_{i},Y_{i+1}^{R},\mathcal{Z}_{i}^{R}\right)\right]
\]
Remark that
\begin{eqnarray*}
\left|f\left(\left[X_{i}\right]_{X},I_{i},Y_{i+1}^{R},\mathcal{Z}_{i}^{R}\right)\right| & \leq & \left|f\left(\left[X_{i}\right]_{X},I_{i},Y_{i+1}^{R},\mathcal{Z}_{i}^{R}\right)-f\left(0,0,0,0\right)\right|+\left|f\left(0,0,0,0\right)\right|\\
 & \leq & L_{f}\left(\left|\left[X_{i}\right]_{X}\right|+\left|I_{i}\right|+\left|Y_{i+1}^{R}\right|+\left|\mathcal{Z}_{i}^{R}\right|\right)+\left|f\left(0,0,0,0\right)\right|\\
 & \leq & C_{f}\left(R_{X}\right)+L_{f}\left(\left|Y_{i+1}^{R}\right|+\left|\mathcal{Z}_{i}^{R}\right|\right)
\end{eqnarray*}
where $C_{f}\left(R_{X}\right):=L_{f}\left(\left|R_{X}\right|+\left|\bar{A}\right|\right)+\left|f\left(0,0,0,0\right)\right|$.
Hence
\begin{eqnarray*}
\left(\mathcal{Y}_{i}^{R}\right)^{2} & \leq & \left(1+\gamma\Delta_{i}\right)\mathbb{E}_{i}\left[Y_{i+1}^{R}\right]^{2}+3\left(\Delta_{i}+\frac{1}{\gamma}\right)\Delta_{i}\left(C_{f}^{2}\left(R_{X}\right)+L_{f}^{2}\mathbb{E}_{i}\left[\left(Y_{i+1}^{R}\right)^{2}\right]+L_{f}^{2}\mathbb{E}_{i}\left[\left|\mathcal{Z}_{i}^{R}\right|\right]\right)\\
 & \leq & \left(\Delta_{i}+\frac{1}{\gamma}\right)\left(\mathbb{E}_{i}\left[Y_{i+1}^{R}\right]^{2}\left(\gamma-3qL_{f}^{2}\right)+\mathbb{E}_{i}\left[\left(Y_{i+1}^{R}\right)^{2}\right]3L_{f}^{2}\left(\Delta_{i}-q\right)+3C_{f}^{2}\left(R_{X}\right)\Delta_{i}\right)
\end{eqnarray*}
Thus, for every $\gamma\geq3qL_{f}^{2}$, one can group together the
terms involving $\mathbb{E}_{i}\left[Y_{i+1}^{R}\right]^{2}$ and
$\mathbb{E}_{i}\left[\left(Y_{i+1}^{R}\right)^{2}\right]$ using Jensen's
inequality:
\begin{eqnarray*}
\left(\mathcal{Y}_{i}^{R}\right)^{2} & \leq & \left(1+\theta\left(3,\gamma\right)\Delta_{i}\right)\mathbb{E}_{i}\left[\left(Y_{i+1}^{R}\right)^{2}\right]+3\left(\left|\pi\right|+\frac{1}{\gamma}\right)C_{f}^{2}\left(R_{X}\right)\Delta_{i}
\end{eqnarray*}
where $\theta\left(c,\gamma\right):=\gamma+cL_{f}^{2}\left(\left|\pi\right|+\frac{1}{\gamma}\right)$.
Finally:
\begin{equation}
\left(Y_{i}^{R}\right)^{2}\leq\esssup_{a\in A}\mathbb{E}_{i,a}\left[\left(Y_{i+1}^{R}\right)^{2}\right]\left(1+\theta\left(3,\gamma\right)\Delta_{i}\right)+3\left(\left|\pi\right|+\frac{1}{\gamma}\right)C_{f}^{2}\left(R_{X}\right)\Delta_{i}\label{eq:asboundRec}
\end{equation}
Using equations \eqref{eq:asboundFinal} and \eqref{eq:asboundRec},
one obtains by induction that:
\begin{equation}
\left(Y_{i}^{R}\right)^{2}\leq\Gamma_{i}^{N-1}\left(3,\gamma\right)C_{g}^{2}\left(R_{X}\right)+3\left(\left|\pi\right|+\frac{1}{\gamma}\right)C_{f}^{2}\left(R_{X}\right)\sum_{k=i}^{N-1}\Gamma_{i}^{k}\left(3,\gamma\right)\Delta_{k}\label{eq:asbound}
\end{equation}
where $\Gamma_{i}^{j}\left(c,\gamma\right):=\Pi_{k=i}^{j}\left(1+\theta\left(c,\gamma\right)\Delta_{k}\right)$.
Finally remark that $\forall c,\gamma>0$ 
\[
\ln\left(\Gamma_{i}^{j}\left(c,\gamma\right)\right)=\sum_{k=i}^{j}\ln\left(1+\theta\left(c,\gamma\right)\Delta_{k}\right)\leq\sum_{k=i}^{j}\theta\left(c,\gamma\right)\Delta_{k}=\theta\left(c,\gamma\right)\left(t_{j+1}-t_{i}\right)
\]
Thus 
\begin{equation}
\Gamma_{i}^{j}\left(c,\gamma\right)\leq\exp\left(\theta\left(c,\gamma\right)\left(t_{j+1}-t_{i}\right)\right)\label{eq:GammaBound}
\end{equation}
And
\begin{eqnarray}
\sum_{k=i}^{N-1}\Gamma_{i}^{k}\left(c,\gamma\right)\Delta_{k} & \leq & \sum_{k=i}^{N-1}e^{\theta\left(c,\gamma\right)\left(t_{j+1}-t_{i}\right)}\Delta_{k}\leq e^{\theta\left(c,\gamma\right)\left|\pi\right|}\sum_{k=i}^{N-1}e^{\theta\left(c,\gamma\right)\left(t_{j}-t_{i}\right)}\Delta_{k}\nonumber \\
 & \leq & e^{\theta\left(c,\gamma\right)\left|\pi\right|}\int_{t_{i}}^{t_{N}}e^{\theta\left(c,\gamma\right)\left(t-t_{i}\right)}dt=\frac{e^{\theta\left(c,\gamma\right)\left|\pi\right|}}{\theta\left(c,\gamma\right)}\left(e^{\theta\left(c,\gamma\right)\left(t_{N}-t_{i}\right)}-1\right)\label{eq:GammaIntBound}
\end{eqnarray}
Finally, combine equations \eqref{eq:asbound}, \eqref{eq:GammaBound}
and \eqref{eq:GammaIntBound} with $c=3$ and $\gamma\geq3qL_{f}^{2}$
to obtain the following a.s. bound for $Y_{i}^{R}$ :
\begin{eqnarray*}
\left(Y_{i}^{R}\right)^{2} & \leq & e^{\theta\left(3,\gamma\right)\left(t_{N}-t_{i}\right)}C_{g}^{2}\left(R_{X}\right)+3\left(\left|\pi\right|+\frac{1}{\gamma}\right)C_{f}^{2}\left(R_{X}\right)\frac{e^{\theta\left(3,\gamma\right)\left|\pi\right|}}{\theta\left(3,\gamma\right)}\left(e^{\theta\left(3,\gamma\right)\left(t_{N}-t_{i}\right)}-1\right)\\
 & \leq & e^{\theta\left(3,\gamma\right)T}\left\{ C_{g}^{2}\left(R_{X}\right)+3\frac{e^{\theta\left(3,\gamma\right)\left|\pi\right|}}{\theta\left(3,\gamma\right)}\left(\left|\pi\right|+\frac{1}{\gamma}\right)C_{f}^{2}\left(R_{X}\right)\right\} 
\end{eqnarray*}
In particular, for $c=3$ and $\gamma=3qL_{f}^{2}$:
\begin{eqnarray*}
\left(Y_{i}^{R}\right)^{2} & \leq & e^{CT}\left\{ C_{g}^{2}\left(R_{X}\right)+\frac{e^{C\left|\pi\right|}}{L_{f}^{2}}C_{f}^{2}\left(R_{X}\right)\right\} =:C_{y}^{2}
\end{eqnarray*}
where $C:=3L_{f}^{2}\left(q+\left|\pi\right|\right)+\frac{1}{q}$.
The same inequality holds for $\left(\mathcal{Y}_{i}^{R}\right)^{2}$.
For $\mathcal{Z}_{i}^{R}$, use the Cauchy-Schwarz inequality to obtain:
\[
\left(\mathcal{Z}_{i}^{R}\right)^{2}\leq\frac{q}{\Delta_{i}}\mathbb{E}_{i}\left[\left(Y_{i+1}^{R}\right)^{2}\right]\leq\frac{q}{\Delta_{i}}C_{y}^{2}=:C_{z}^{2}
\]
and the same inequality holds for $\left(Z_{i}^{R}\right)^{2}$.\end{proof}
\begin{lem}
For $R>0$, define $\mathcal{T}_{R}=\mathbb{E}\left[\left(\mathcal{N}-\left(-R\right)\vee\mathcal{N}\wedge R\right)^{2}\right]$
where $\mathcal{N}$ is a Gaussian random variable with mean $0$
and variance $1$. Then:
\[
\mathcal{T}_{R}\leq\sqrt{\frac{2}{\pi}}\frac{1}{R}e^{-\frac{R^{2}}{2}}
\]
\end{lem}
\begin{proof}
Developing the square yields 
\[
\mathcal{T}_{R}=2R^{2}\mathbb{P}\left(\mathcal{N}>R\right)-4R\mathbb{E}\left[\mathcal{N}\mathbf{1}\left\{ \mathcal{N}>R\right\} \right]+2\mathbb{E}\left[\mathcal{N}^{2}\mathbf{1}\left\{ \mathcal{N}>R\right\} \right]
\]
Then the two expectations can be explicited as follows
\begin{eqnarray*}
\mathbb{E}\left[\mathcal{N}\mathbf{1}\left\{ \mathcal{N}>R\right\} \right] & = & \frac{e^{-\frac{R^{2}}{2}}}{\sqrt{2\pi}}\\
\mathbb{E}\left[\mathcal{N}^{2}\mathbf{1}\left\{ \mathcal{N}>R\right\} \right] & = & \frac{R}{\sqrt{2\pi}}e^{-\frac{R^{2}}{2}}+\mathbb{P}\left(\mathcal{N}>R\right)
\end{eqnarray*}
Finally, the use of Mill's ratio inequality $\mathbb{P}\left(\mathcal{N}>R\right)<\frac{1}{R}\frac{e^{-\frac{R^{2}}{2}}}{\sqrt{2\pi}}$
concludes the proof.
\end{proof}
Then, we can estimate bounds between the BSDEs \eqref{eq:discreteBSDE}
and \eqref{eq:localizedBSDE}.
\begin{prop}
The following bounds hold:
\begin{eqnarray*}
\left(Y_{i}-Y_{i}^{R}\right)^{2} & \leq & e^{CT}\left\{ L_{g}\left(\left|\Delta X_{N}\right|^{2}\right)^{*}+C\sum_{k=i}^{N-1}\Delta_{k}\left(\left|\Delta X_{k}\right|^{2}\right)^{*}+2qCTC_{y}^{2}\mathcal{T}_{R_{w}}\right\} 
\end{eqnarray*}
where $C:=3L_{f}^{2}\left(2q+\left|\pi\right|\right)+\frac{1}{2q}$\textup{,
and $\left(\left|\Delta X_{k}\right|^{2}\right)^{*}$, $k\geq i$,
is the solution of the following linear constrained BSDE:
\[
\begin{cases}
Y_{k} & =\left(X_{k}-\left[X_{k}\right]_{X}\right)^{2}\\
Y_{j} & =\esssup_{a\in A}\mathbb{E}_{j,a}\left[Y_{j+1}\right]\,,\,\, j=k-1,\ldots,i
\end{cases}
\]
 }\end{prop}
\begin{proof}
Define $\Delta X_{i}=X_{i}-\left[X_{i}\right]_{X}$, $\Delta Y_{i}=Y_{i}-Y_{i}^{R}$,
$\Delta\mathcal{Y}_{i}=\mathcal{Y}_{i}-\mathcal{Y}_{i}^{R}$, $\Delta Z_{i}=Z_{i}-Z_{i}^{R}$
and $\Delta\mathcal{Z}_{i}=\mathcal{Z}_{i}-\mathcal{Z}_{i}^{R}$. 
First
\[
\left|\Delta Y_{N}\right|=\left|g\left(X_{N}\right)-g\left(\left[X_{N}\right]_{X}\right)\right|\leq L_{g}\left|\Delta X_{N}^{p}\right|
\]
Then
\begin{eqnarray*}
\Delta_{i}\Delta\mathcal{Z}_{i} & = & \mathbb{E}_{i}\left[Y_{i+1}\Delta W_{i}^{\top}-Y_{i+1}^{R}\left[\Delta W_{i}^{\top}\right]_{w}\right]\\
 & = & \mathbb{E}_{i}\left[\Delta Y_{i+1}\Delta W_{i}^{\top}+Y_{i+1}^{R}\left\{ \Delta W_{i}-\left[\Delta W_{i}\right]_{w}\right\} ^{\top}\right]\\
 &  & \mathbb{E}_{i}\left[\left(\Delta Y_{i+1}-\mathbb{E}_{i}\left[\Delta Y_{i+1}\right]\right)\Delta W_{i}^{\top}\right]+\mathbb{E}_{i}\left[Y_{i+1}^{R}\left\{ \Delta W_{i}-\left[\Delta W_{i}\right]_{w}\right\} ^{\top}\right]
\end{eqnarray*}
Hence
\[
\Delta_{i}\left(\Delta\mathcal{Z}_{i}\right)^{2}\leq2q\left(\mathbb{E}_{i}\left[\left(\Delta Y_{i+1}\right)^{2}\right]-\mathbb{E}_{i}\left[\Delta Y_{i+1}\right]^{2}\right)+2qC_{y}^{2}\mathcal{T}_{R_{w}}
\]
Then
\[
\Delta\mathcal{Y}_{i}=\mathbb{E}_{i}\left[\Delta Y_{i+1}+\left\{ f\left(X_{i},I_{i},Y_{i+1},\mathcal{Z}_{i}\right)-f\left(\left[X_{i}\right]_{X},I_{i},Y_{i+1}^{R},\mathcal{Z}_{i}^{R}\right)\right\} \Delta_{i}\right]
\]
Using Jensen's inequality and Young's inequality with parameter $\gamma\Delta_{i}$,
$\gamma>0$:
\begin{eqnarray*}
\left(\Delta\mathcal{Y}_{i}\right)^{2} & \leq & \left(1+\gamma\Delta_{i}\right)\mathbb{E}_{i}\left[\Delta Y_{i+1}\right]^{2}+\left(1+\frac{1}{\gamma\Delta_{i}}\right)\Delta_{i}^{2}3L_{f}^{2}\mathbb{E}_{i}\left[\left(\Delta X_{i}\right)^{2}+\left(\Delta Y_{i+1}\right)^{2}+\left(\Delta\mathcal{Z}_{i}\right)^{2}\right]\\
 & \leq & \mathbb{E}_{i}\left[\Delta Y_{i+1}\right]^{2}\left(\Delta_{i}+\frac{1}{\gamma}\right)\left(\gamma-6qL_{f}^{2}\right)+\mathbb{E}_{i}\left[\left(\Delta Y_{i+1}\right)^{2}\right]\left(\Delta_{i}+\frac{1}{\gamma}\right)3L_{f}^{2}\left(\Delta_{i}+2q\right)\\
 &  & +\left(\Delta_{i}+\frac{1}{\gamma}\right)\Delta_{i}3L_{f}^{2}\left\{ \left(\Delta X_{i}\right)^{2}+2qC_{y}^{2}\mathcal{T}_{R_{w}}\right\} 
\end{eqnarray*}
Now, for any $\gamma\geq6qL_{f}^{2}$, one can group together the
terms in $\mathbb{E}_{i}\left[\Delta Y_{i+1}\right]^{2}$ and $\mathbb{E}_{i}\left[\left(\Delta Y_{i+1}\right)^{2}\right]$
using Jensen's inequality:
\[
\left(\Delta\mathcal{Y}_{i}\right)^{2}\leq\mathbb{E}_{i}\left[\left(\Delta Y_{i+1}\right)^{2}\right]\left\{ 1+\theta\left(3,\gamma\right)\Delta_{i}\right\} +3L_{f}^{2}\left(\left|\pi\right|+\frac{1}{\gamma}\right)\Delta_{i}\left\{ \left(\Delta X_{i}\right)^{2}+2qC_{y}^{2}\mathcal{T}_{R_{w}}\right\} 
\]
where, as in Lemma \ref{lem:asBounds}, $\theta\left(c,\gamma\right):=\gamma+cL_{f}^{2}\left(\left|\pi\right|+\frac{1}{\gamma}\right)$.
Hence, using that for any random variables $\Theta$ and $\Theta'$,
\[
\left(\esssup_{a\in A}\mathbb{E}_{i,a}\left[\Theta\right]-\esssup_{a\in A}\mathbb{E}_{i,a}\left[\Theta'\right]\right)^{2}\leq\esssup_{a\in A}\mathbb{E}_{i,a}\left[\left(\Theta-\Theta'\right)^{2}\right]\,,
\]
the following holds:
\[
\left(\Delta Y_{i}\right)^{2}\leq\left\{ 1+\theta\left(3,\gamma\right)\Delta_{i}\right\} \esssup_{a\in A}\mathbb{E}_{i,a}\left[\left(\Delta Y_{i+1}\right)^{2}\right]+3L_{f}^{2}\left(\left|\pi\right|+\frac{1}{\gamma}\right)\Delta_{i}\left\{ \left(\Delta X_{i}\right)^{2}+2qC_{y}^{2}\mathcal{T}_{R_{w}}\right\} 
\]
By induction
\begin{eqnarray*}
\left(\Delta Y_{i}\right)^{2} & \leq & L_{g}\Gamma_{i}^{N-1}\left(3,\gamma\right)\left(\left|\Delta X_{N}\right|^{2}\right)^{*}\\
 &  & +3L_{f}^{2}\left(\left|\pi\right|+\frac{1}{\gamma}\right)\sum_{k=i}^{N-1}\Delta_{k}\Gamma_{i}^{k}\left(3,\gamma\right)\left\{ \left(\left|\Delta X_{k}\right|^{2}\right)^{*}+2qC_{y}^{2}\mathcal{T}_{R_{w}}\right\} 
\end{eqnarray*}
where, as in Lemma \ref{lem:asBounds}, $\Gamma_{i}^{j}\left(c,\gamma\right):=\Pi_{k=i}^{j}\left(1+\theta\left(c,\gamma\right)\Delta_{k}\right)\leq\exp\left(\theta\left(c,\gamma\right)\left(t_{j+1}-t_{i}\right)\right)$
. Finally, take $\gamma=6qL_{f}^{2}$ to obtain the desired bound.
\end{proof}

\subsection{Projections}

In its current form, the scheme \eqref{eq:localizedBSDE} is not readily
implementable, because its conditional expectations cannot be computed
in general. Therefore, there is a need to approximate these conditional
expectations. For handiness and efficiency, we choose, in the spirit
of \cite{Lemor06} and \cite{Gobet11}, to approximate them by empirical
least-squares regression.

First, we will study the impact of the replacement of the conditional
expectations by theoretical least-squares regressions. We will see
that the resulting scheme is not easy to analyze. Therefore, we will
study a stronger version of it, and discuss their practical differences.
As it is already a daunting task for standard BSDEs (cf. \cite{Lemor06}),
and in view of the difficulties already raised at theoretical regression
level, we leave the study of the final replacement of these regressions
by their empirical counterparts for further research.

Hence, for each $i\in\left\{ 0,\ldots,N-1\right\} $, consider $\mathcal{S}_{i}^{Y}$
and $\mathcal{S}_{i}^{Z}=\left\{ \mathcal{S}_{i}^{Z,1},\ldots,\mathcal{S}_{i}^{Z,q}\right\} $
that are non-empty closed convex subsets of $\mathbf{L}_{2}\left(\mathcal{F}_{t_{i}},\mathbb{P}\right)$,
as well as the corresponding projection operators $\mathcal{P}_{i}^{Y}$
and $\mathcal{P}_{i}^{Z}=\left\{ \mathcal{P}_{i}^{Z,1},\ldots,\mathcal{P}_{i}^{Z,q}\right\} $.
Using the above projection operators in lieu of the conditional expectations
in \eqref{eq:localizedBSDE}, we obtain the following approximation
scheme:
\begin{equation}
\begin{cases}
\tilde{Y}_{N}^{R} & =g\left(\left[X_{N}\right]_{X}\right)\\
\Delta_{i}\tilde{\mathcal{Z}}_{i}^{R} & =\left[\mathcal{P}_{i}^{Z}\left(\tilde{Y}_{i+1}^{R}\left[\Delta W_{i}^{\top}\right]_{w}\right)\right]_{i,z}\\
\tilde{\mathcal{Y}}_{i}^{R} & =\left[\mathcal{P}_{i}^{Y}\left(\tilde{Y}_{i+1}^{R}+f\left(\left[X_{i}\right]_{X},I_{i},\tilde{Y}_{i+1}^{R},\tilde{\mathcal{Z}}_{i}^{R}\right)\Delta_{i}\right)\right]_{y}\\
\tilde{Y}_{i}^{R} & =\esssup_{a\in A}\mathbb{E}_{i,a}\left[\tilde{\mathcal{Y}}_{i}^{R}\right]
\end{cases}\label{eq:projectedBSDE}
\end{equation}
where $\left[.\right]_{i,z}:=-\Delta_{i}C_{z}\wedge.\vee\Delta_{i}C_{z}$
and $\left[.\right]_{y}:=-C_{y}\wedge.\vee C_{y}$ are truncation
operators that ensure that the a.s. upper bounds for $\left(Y^{R},Z^{R}\right)$
from Lemma \ref{lem:asBounds} will also hold for $\left(\tilde{Y}^{R},\tilde{Z}^{R}\right)$.

To be more specific, choose the subsets $\mathcal{S}_{i}^{Y}$ and
$\mathcal{S}_{i}^{Z}$ as follows:
\begin{align*}
\mathcal{S}_{i}^{Y} & =\left\{ \lambda.p_{i}^{Y}\left(X_{i},I_{i}\right)\,;\,\lambda\in\mathbb{R}^{B_{i}^{Y}}\right\} \\
\mathcal{S}_{i}^{Z,k} & =\left\{ \lambda.p_{i}^{Z,k}\left(X_{i},I_{i}\right)\,;\,\lambda\in\mathbb{R}^{B_{i}^{Z,k}}\right\} \,,\, k=1,\ldots,q
\end{align*}
where $p_{i}^{Y}=\left(p_{i,1}^{Y},\ldots,p_{i,B_{i}^{Y}}^{Y}\right)^{\top}$,
$B_{i}^{Y}\geq1$, and $p_{i}^{Z,k}=\left(p_{i,1}^{Z,k},\ldots,p_{i,B_{i}^{Z,k}}^{,Z,k}\right)^{\top}$,
$B_{i}^{Z,k}\geq1$, are predefined sets of deterministic functions
from $\mathbb{R}^{d}\times\mathbb{R}^{q}$ into $\mathbb{R}$. Hence,
for any random variable $U$ in $\mathbf{L}_{2}\left(\mathcal{F}_{T},\mathbb{P}\right)$,
$\mathcal{P}_{i}^{Y}\left(U\right)$ is defined as follows:
\begin{align}
\hat{\lambda}_{i}^{Y}\left(U\right) & :=\arg\inf_{\lambda\in\mathbb{R}^{B_{i}^{Y}}}\mathbb{E}\left[\left(\lambda.p_{i}^{Y}\left(X_{i},I_{i}\right)-U\right)^{2}\right]\label{eq:regrCoef}\\
\mathcal{P}_{i}^{Y}\left(U\right) & :=\hat{\lambda}_{i}^{Y}\left(U\right).p_{i}^{Y}\left(X_{i},I_{i}\right)\nonumber 
\end{align}
and $\mathcal{P}_{i}^{Z}\left(U\right)$ is defined in a similar manner.
With these notations, the scheme \eqref{eq:projectedBSDE} can be
explicited further as follows:
\begin{equation}
\begin{cases}
\tilde{Y}_{N}^{R} & =g\left(\left[X_{N}\right]_{X}\right)\\
\Delta_{i}\tilde{\mathcal{Z}}_{i}^{R} & =\left[\hat{\lambda}_{i}^{Z}\left(\tilde{Y}_{i+1}^{R}\left[\Delta W_{i}^{\top}\right]_{w}\right).p_{i}^{Z}\left(X_{i},I_{i}\right)\right]_{i,z}\\
\tilde{Y}_{i}^{R} & =\esssup_{a\in\mathcal{A}_{i}}\left[\hat{\lambda}_{i}^{Y}\left(\tilde{Y}_{i+1}^{R}+f\left(\left[X_{i}\right]_{X},I_{i},\tilde{Y}_{i+1}^{R},\tilde{\mathcal{Z}}_{i}^{R}\right)\Delta_{i}\right).p_{i}^{Y}\left(X_{i},a\right)\right]_{y}\,,
\end{cases}\label{eq:projectedBSDE2}
\end{equation}
where $\mathcal{A}_{i}$ is the set of $\sigma\left(X_{i}\right)$-measurable
random variables taking values in $A$. Now, we would like to analyze
the error between $\left(Y^{R},Z^{R}\right)$ and $\left(\tilde{Y}^{R},\tilde{Z}^{R}\right)$.
Unfortunately, in spite of the simplicity of the scheme \eqref{eq:projectedBSDE2},
this analysis is made strenuous by the fact that $\tilde{Y}_{i}^{R}$
is not itself a projection, as it combines regression coefficients
computed using the random variable $I_{i}$ and regression functions
valued at another random variable $a$. This prevents the analysis
from taking advantage of standard tools to deal with least-squares
regressions. For comparison, consider the following alternative scheme:
\begin{equation}
\begin{cases}
\hat{Y}_{N}^{R} & =g\left(\left[X_{N}\right]_{X}\right)\\
\Delta_{i}\hat{\mathcal{Z}}_{i,a}^{R} & =\left[\hat{\lambda}_{i,a}^{Z}\left(\hat{Y}_{i+1}^{R}\left[\Delta W_{i}^{\top}\right]_{w}\right).p_{i}^{Z}\left(X_{i},I_{i}\right)\right]_{i,z}\,,\, a\in\mathcal{A}_{i}\\
\hat{Y}_{i}^{R} & =\esssup_{a\in\mathcal{A}_{i}}\left[\hat{\lambda}_{i,a}^{Y}\left(\hat{Y}_{i+1}^{R}+f\left(\left[X_{i}\right]_{X},I_{i},\hat{Y}_{i+1}^{R},\hat{\mathcal{Z}}_{i,a}^{R}\right)\Delta_{i}\right).p_{i}^{Y}\left(X_{i},a\right)\right]_{y}
\end{cases}\label{eq:projectedBSDE*}
\end{equation}
where, unlike equation \eqref{eq:regrCoef}, the regression coefficients
$\hat{\lambda}_{i,a}^{Y}$ are computed as follows:
\begin{align}
\hat{\lambda}_{i,a}^{Y}\left(U\right) & :=\arg\inf_{\lambda\in\mathbb{R}^{B_{i}^{Y}}}\mathbb{E}\left[\left(\lambda.p_{i}^{Y}\left(X_{i},a\right)-U_{a}\right)^{2}\right]\label{eq:regrCoefA}\\
\mathcal{P}_{i,a}^{Y}\left(U\right) & :=\hat{\lambda}_{i,a}^{Y}\left(U\right).p_{i}^{Y}\left(X_{i},a\right)\nonumber 
\end{align}
for every $U\in\mathbf{L}_{2}\left(\mathcal{F}_{T},\mathbb{P}\right)$
and $a\in\mathcal{A}_{i}$, where $U_{a}$ corresponds to the conditional
random variable $U\left|\left\{ I_{i}=a\right\} \right.$. $\mathcal{P}_{i,a}^{Z}\left(U\right)$
is defined in a similar manner. Remark that $\mathcal{P}_{i,I_{i}}^{.}\left(U\right)=\mathcal{P}_{i}^{.}\left(U\right)$,
and that $\mathcal{P}_{i,a}^{.}\left(U_{a}\right)=\mathcal{P}_{i,a}^{.}\left(U\right)$.
With this new scheme, the estimated regression coefficients are changed
along with the strategy $a$ when computing the optimal strategy.
Therefore, compared with the scheme \eqref{eq:projectedBSDE2}, the
implementation of an empirical version of the scheme \eqref{eq:projectedBSDE*}
is much more involved, as it may require, for the same time step,
many regressions involving several random variables $a$ different
from $I_{i}$ (which is used to simulate the forward process). However,
these modifications ease considerably the analysis of the impact of
the projections compared with $\left(Y^{R},Z^{R}\right)$ as shown below
in the remaining of this subsection.

First, the scheme \eqref{eq:projectedBSDE*} can be written as follows:
\begin{equation}
\begin{cases}
\hat{Y}_{N}^{R} & =g\left(\left[X_{N}\right]_{X}\right)\\
\Delta_{i}\hat{\mathcal{Z}}_{i,a}^{R} & =\left[\mathcal{P}_{i,a}^{Z}\left(\hat{Y}_{i+1}^{R}\left[\Delta W_{i}^{\top}\right]_{w}\right)\right]_{i,z}\\
\hat{\mathcal{Y}}_{i,a}^{R} & =\left[\mathcal{P}_{i,a}^{Y}\left(\hat{Y}_{i+1}^{R}+f\left(\left[X_{i}\right]_{X},I_{i},\hat{Y}_{i+1}^{R},\hat{\mathcal{Z}}_{i,a}^{R}\right)\Delta_{i}\right)\right]_{y}\\
\hat{Y}_{i}^{R} & =\esssup_{a\in A}\hat{\mathcal{Y}}_{i,a}^{R}
\end{cases}\label{eq:projectedBSDE2*}
\end{equation}
Then, we recall below some useful properties of the projection operators
$\mathcal{P}_{i,a}^{.}$.
\begin{lem}
For any fixed $a\in\mathcal{A}_{i}$: 
\begin{eqnarray}
\mathcal{P}_{i,a}^{.}\left(U\right) & = & \mathcal{P}_{i,a}^{.}\left(\mathbb{E}_{i,a}\left[U\right]\right)\,\,,\,\forall U\in\mathbf{L}_{2}\left(\mathcal{F}_{t_{i}},\mathbb{P}\right)\label{lem:projEsp}\\
\mathbb{E}\left[\left(\mathcal{P}_{i,a}^{.}\left(U\right)-\mathcal{P}_{i,a}^{.}\left(V\right)\right)^{2}\right] & \leq & \mathbb{E}\left[\left(U_{a}-V_{a}\right)^{2}\right]\,\,,\,\forall U,V\,\mathrm{in}\,\mathbf{L}_{2}\left(\mathcal{F}_{T},\mathbb{P}\right).\label{lem:projLip}
\end{eqnarray}
\end{lem}
\begin{proof}
The proof can be found in \cite{Gobet11}.
\end{proof}
We now assess the error between $\left(Y^{R},Z^{R}\right)$ and $\left(\hat{Y}^{R},\hat{Z}^{R}\right)$.
\begin{prop}
{[}projection error{]} The following bound holds:
\begin{align*}
\mathbb{E}\left[\left|Y_{i}^{R}-\hat{Y}{}_{i}^{R}\right|^{2}\right]\,,\,\Delta_{i}\mathbb{E}\left[\left|Z_{i}^{R}-\hat{Z}{}_{i}^{R}\right|^{2}\right]\leq & e^{C\left(T-t_{i}\right)}\sum_{k=i}^{N-1}\left\{ \mathbb{E}\left[\left(\left|\Delta\mathcal{P}\mathcal{Y}_{k}\right|^{2}\right)^{*}\right]+C\Delta_{k}\mathbb{E}\left[\left(\left|\Delta\mathcal{P}\mathcal{Z}_{k}\right|^{2}\right)^{*}\right]\right\} 
\end{align*}
where $C:=2L_{f}^{2}\left(\left|\pi\right|+q\right)+\frac{1}{q}$,
and $\left(\left|\Delta\mathcal{P}\mathcal{Y}_{k}\right|^{2}\right)^{*}$
(resp. \textup{$\left(\left|\Delta\mathcal{P}\mathcal{Z}_{k}\right|^{2}\right)^{*}$),
$k\geq i$, is solution of the linear constrained BSDE:}
\textup{
\[
\begin{cases}
Y_{k} & =\esssup_{a\in A}\mathbb{E}_{k,a}\left[\left|\mathcal{Y}_{k}^{R}-\mathcal{P}_{k}^{Y}\left(\mathcal{Y}_{k}^{R}\right)\right|^{2}\right]\,,\,\,(\mathrm{resp}.\,\esssup_{a\in A}\mathbb{E}_{k,a}\left[\left|\mathcal{Z}_{k}^{R}-\mathcal{P}_{k}^{Z}\left(\mathcal{Z}_{k}^{R}\right)\right|^{2}\right])\\
Y_{j} & =\esssup_{a\in A}\mathbb{E}_{j,a}\left[Y_{j+1}\right]\,,\,\, j=k-1,\ldots,i
\end{cases}
\]
}Moreover, the same upper bound holds for $\mathbb{E}\left[\esssup_{a\in A}\left|\mathcal{Y}_{i,a}^{R}-\hat{\mathcal{Y}}{}_{i,a}^{R}\right|^{2}\right]$
and $\Delta_{i}\mathbb{E}\left[\esssup_{a\in A}\left|\mathcal{Z}_{i,a}^{R}-\hat{\mathcal{Z}}{}_{i,a}^{R}\right|^{2}\right]$.\end{prop}
\begin{proof}
Fix $a\in\mathcal{A}_{i}$. Define $\Delta Y_{i}^{R}=Y_{i}^{R}-\hat{Y}_{i}^{R}$,
$\Delta\mathcal{Y}_{i,a}^{R}=\mathcal{Y}_{i,a}^{R}-\hat{\mathcal{Y}}_{i,a}^{R}$,
$\Delta Z_{i}^{R}=Z_{i}^{R}-\hat{Z}_{i}^{R}$ and $\Delta\mathcal{Z}_{i,a}^{R}=\mathcal{Z}_{i,a}^{R}-\hat{\mathcal{Z}}_{i,a}^{R}$,
where, as in equation \eqref{eq:regrCoefA}, $\mathcal{Y}_{i,a}^{R}$
(resp. $\mathcal{Z}_{i,a}^{R}$) stands for the conditional variable
$\mathcal{Y}_{i}^{R}\left|\left\{ I_{i}=a\right\} \right.$ (resp.
$\mathcal{Z}_{i}^{R}\left|\left\{ I_{i}=a\right\} \right.$).

First, using that $\Delta_{i}\mathcal{Z}_{i,a}^{R}=\left[\Delta_{i}\mathcal{Z}_{i,a}^{R}\right]_{i,z}$
and the $1$-Lipschitz property of $\left[.\right]_{i,z}$:
\[
\left|\Delta_{i}\Delta\mathcal{Z}_{i,a}^{R}\right|^{2}\leq\left|\Delta_{i}\mathcal{Z}_{i,a}^{R}-\mathcal{P}_{i,a}^{Z}\left(\hat{Y}_{i+1}^{R}\left[\Delta W_{i}^{\top}\right]_{w}\right)\right|^{2}
\]
Using Pythagoras' theorem:
\[
\mathbb{E}\left[\left|\Delta_{i}\Delta\mathcal{Z}_{i,a}^{R}\right|^{2}\right]=\mathbb{E}\left[\left|\Delta_{i}\mathcal{Z}_{i,a}^{R}-\mathcal{P}_{i,a}^{Z}\left(\Delta_{i}\mathcal{Z}_{i,a}^{R}\right)\right|^{2}\right]+\mathbb{E}\left[\left|\mathcal{P}_{i,a}^{Z}\left(\Delta_{i}\mathcal{Z}_{i,a}^{R}\right)-\mathcal{P}_{i,a}^{Z}\left(\hat{Y}_{i+1}^{R}\left[\Delta W_{i}^{\top}\right]_{w}\right)\right|^{2}\right]
\]
where, using equation \eqref{lem:projEsp}:
\begin{align*}
 & \mathcal{P}_{i,a}^{Z}\left(\hat{Y}_{i+1}^{R}\left[\Delta W_{i}^{\top}\right]_{w}\right)=\mathcal{P}_{i,a}^{Z}\left(\mathbb{E}_{i,a}\left[\hat{Y}_{i+1}^{R}\left[\Delta W_{i}^{\top}\right]_{w}\right]\right)\\
= & \mathcal{P}_{i,a}^{Z}\left(\mathbb{E}_{i,a}\left[\left(\hat{Y}_{i+1}^{R}-\mathbb{E}_{i,a}\left[\hat{Y}_{i+1}^{R}\right]\right)\left[\Delta W_{i}^{\top}\right]_{w}\right]\right)
\end{align*}
Then, using equation \eqref{lem:projLip}:
\begin{align*}
\mathbb{E}\left[\left|\mathcal{P}_{i,a}^{Z}\left(\Delta_{i}\mathcal{Z}_{i,a}^{R}\right)-\mathcal{P}_{i,a}^{Z}\left(\hat{Y}_{i+1}^{R}\left[\Delta W_{i}^{\top}\right]_{w}\right)\right|^{2}\right] & \leq\mathbb{E}\left[\left|\mathbb{E}_{i,a}\left[\Delta_{i}\mathcal{Z}_{i,a}^{R}\right]-\mathbb{E}_{i,a}\left[\left(\hat{Y}_{i+1}^{R}-\mathbb{E}_{i,a}\left[\hat{Y}_{i+1}^{R}\right]\right)\left[\Delta W_{i}^{\top}\right]_{w}\right]\right|^{2}\right]\\
 & =\mathbb{E}\left[\left|\mathbb{E}_{i,a}\left[\left(\Delta Y_{i+1}^{R}-\mathbb{E}_{i,a}\left[\Delta Y_{i+1}^{R}\right]\right)\left[\Delta W_{i}^{\top}\right]_{w}\right]\right|^{2}\right]\\
 & \leq q\Delta_{i}\mathbb{E}\left[\mathbb{E}_{i,a}\left[\left(\Delta Y_{i+1}^{R}\right)^{2}\right]-\mathbb{E}_{i,a}\left[\Delta Y_{i+1}^{R}\right]^{2}\right]
\end{align*}
To sum up for the $Z$ component:
\[
\Delta_{i}\mathbb{E}\left[\left|\Delta\mathcal{Z}_{i,a}^{R}\right|^{2}\right]\leq\Delta_{i}\mathbb{E}\left[\left|\mathcal{Z}_{i,a}^{R}-\mathcal{P}_{i,a}^{Z}\left(\mathcal{Z}_{i,a}^{R}\right)\right|^{2}\right]+q\mathbb{E}\left[\mathbb{E}_{i,a}\left[\left(\Delta Y_{i+1}^{R}\right)^{2}\right]-\mathbb{E}_{i,a}\left[\Delta Y_{i+1}^{R}\right]^{2}\right]
\]
For the $Y$ component, start similarly by using the $1$-Lipschitz
property of $\left[.\right]_{y}$ and Pythagoras' theorem:
\[
\mathbb{E}\left[\left|\Delta\mathcal{Y}_{i,a}^{R}\right|^{2}\right]=\mathbb{E}\left[\left(\mathcal{Y}_{i,a}^{R}-\mathcal{P}_{i,a}^{Y}\left(\mathcal{Y}_{i,a}^{R}\right)\right)^{2}\right]+\mathbb{E}\left[\left(\mathcal{P}_{i,a}^{Y}\left(\mathcal{Y}_{i,a}^{R}\right)-\mathcal{P}_{i,a}^{Y}\left(\hat{Y}_{i+1}^{R}+f\left(\left[X_{i}\right]_{X},I_{i},\hat{Y}_{i+1}^{R},\hat{\mathcal{Z}}_{i,a}^{R}\right)\Delta_{i}\right)\right)^{2}\right]
\]
And then, using again equations \eqref{lem:projEsp}, \eqref{lem:projLip},
Jensen's inequality and Young's inequality with parameter $\gamma\Delta_{i}$,
$\gamma>0$:
\begin{align*}
 & \mathbb{E}\left[\left(\mathcal{P}_{i,a}^{Y}\left(\mathcal{Y}_{i,a}^{R}\right)-\mathcal{P}_{i,a}^{Y}\left(\hat{Y}_{i+1}^{R}+f\left(\left[X_{i}\right]_{X},I_{i},\hat{Y}_{i+1}^{R},\hat{\mathcal{Z}}_{i,a}^{R}\right)\Delta_{i}\right)\right)^{2}\right]\\
 & \leq\mathbb{E}\left[\left(\mathbb{E}_{i,a}\left[\Delta Y_{i+1}^{R}+L_{f}\left(\left|\Delta Y_{i+1}^{R}\right|+\left|\Delta\mathcal{Z}_{i,a}^{R}\right|\right)\Delta_{i}\right]\right)^{2}\right]\\
 & \leq\mathbb{E}\left[\left(1+\gamma\Delta_{i}\right)\mathbb{E}_{i,a}\left[\Delta Y_{i+1}^{R}\right]^{2}+\left(1+\frac{1}{\gamma\Delta_{i}}\right)\Delta_{i}^{2}L_{f}^{2}2\left\{ \mathbb{E}_{i,a}\left[\left(\Delta Y_{i+1}^{R}\right)^{2}\right]+\mathbb{E}_{i,a}\left[\left|\Delta\mathcal{Z}_{i,a}^{R}\right|^{2}\right]\right\} \right]\\
 & \leq\left(\Delta_{i}+\frac{1}{\gamma}\right)\mathbb{E}\left[\left(\gamma-2qL_{f}^{2}\right)\mathbb{E}_{i,a}\left[\Delta Y_{i+1}^{R}\right]^{2}+2L_{f}^{2}\left(\Delta_{i}+q\right)\mathbb{E}_{i,a}\left[\left(\Delta Y_{i+1}^{R}\right)^{2}\right]+2L_{f}^{2}\Delta_{i}\left|\mathcal{Z}_{i,a}^{R}-\mathcal{P}_{i,a}^{Z}\left(\mathcal{Z}_{i,a}^{R}\right)\right|^{2}\right]
\end{align*}
For all $\gamma\geq2qL_{f}^{2}$, one can group together the terms
involving $\mathbb{E}_{i,a}\left[Y_{i+1}^{R}\right]^{2}$ and $\mathbb{E}_{i,a}\left[\left(Y_{i+1}^{R}\right)^{2}\right]$
using Jensen's inequality:
\begin{align}
\mathbb{E}\left[\left|\Delta\mathcal{Y}_{i,a}^{R}\right|^{2}\right]\leq & \mathbb{E}\left[\left(\mathcal{Y}_{i,a}^{R}-\mathcal{P}_{i,a}^{Y}\left(\mathcal{Y}_{i,a}^{R}\right)\right)^{2}\right]+\left(1+\theta\left(2,\gamma\right)\Delta_{i}\right)\mathbb{E}\left[\mathbb{E}_{i,a}\left[\left(\Delta Y_{i+1}^{R}\right)^{2}\right]\right]\nonumber \\
 & +2L_{f}^{2}\left(\left|\pi\right|+\frac{1}{\gamma}\right)\Delta_{i}\mathbb{E}\left[\left|\mathcal{Z}_{i,a}^{R}-\mathcal{P}_{i,a}^{Z}\left(\mathcal{Z}_{i,a}^{R}\right)\right|^{2}\right]\nonumber \\
\leq & \mathbb{E}\left[\esssup_{a\in A}\left(\mathcal{Y}_{i,a}^{R}-\mathcal{P}_{i,a}^{Y}\left(\mathcal{Y}_{i,a}^{R}\right)\right)^{2}\right]+\left(1+\theta\left(2,\gamma\right)\Delta_{i}\right)\mathbb{E}
\left[\esssup_{a\in A}\mathbb{E}_{i,a}\left[\left(\Delta Y_{i+1}^{R}\right)^{2}\right]\right]\label{eq:DeltaYA}\\
 & +2L_{f}^{2}\left(\left|\pi\right|+\frac{1}{\gamma}\right)\Delta_{i}\mathbb{E}\left[\esssup_{a\in A}\left|\mathcal{Z}_{i,a}^{R}-\mathcal{P}_{i,a}^{Z}\left(\mathcal{Z}_{i,a}^{R}\right)\right|^{2}\right]\nonumber 
\end{align}
where $\theta\left(c,\gamma\right)=\gamma+cL_{f}^{2}\left(\left|\pi\right|+\frac{1}{\gamma}\right)$. 

Therefore, as equation \eqref{eq:DeltaYA} is true for every $a\in\mathcal{A}_{i}$
on its left-hand side, and as $\left|\Delta Y_{i+1}^{R}\right|^{2}\leq\esssup_{a\in A}\left|\Delta\mathcal{Y}_{i+1,a}^{R}\right|^{2}$,
the following holds by induction:
\begin{align*}
\mathbb{E}\left[\esssup_{a\in A}\left|\Delta\mathcal{Y}_{i,a}^{R}\right|^{2}\right]\leq & \sum_{k=i}^{N-1}\Gamma_{i}^{k}\left(2,\gamma\right)\left\{ \mathbb{E}\left[\left(\left|\Delta\mathcal{P}\mathcal{Y}_{k}\right|^{2}\right)^{*}\right]+2L_{f}^{2}\left(\left|\pi\right|+\frac{1}{\gamma}\right)\Delta_{k}\mathbb{E}\left[\left(\left|\Delta\mathcal{P}\mathcal{Z}_{k}\right|^{2}\right)^{*}\right]\right\} 
\end{align*}
where $\Gamma_{i}^{j}\left(c,\gamma\right)=\Pi_{k=i}^{j}\left(1+\theta\left(c,\gamma\right)\Delta_{k}\right)\leq\exp\left(\theta\left(c,\gamma\right)\left(t_{j+1}-t_{i}\right)\right)$.
Finally, take $\gamma=2qL_{f}^{2}$ to obtain the desired bound for
$\left|\Delta\mathcal{Y}_{i,a}^{R}\right|^{2}$. Moreover, as $\left|\Delta Y_{i}^{R}\right|^{2}\leq\esssup_{a\in A}\left|\Delta\mathcal{Y}_{i,a}^{R}\right|^{2}$,
the same bound holds for $\left|\Delta Y_{i}^{R}\right|^{2}$. For
the bound on $\left|\Delta Z_{i}^{R}\right|^{2}$, use that:
\[
\Delta_{i}\mathbb{E}\left[\left|\Delta Z_{i}^{R}\right|^{2}\right]\leq\Delta_{i}\mathbb{E}\left[\esssup_{a\in A}\left|\Delta\mathcal{Z}_{i,a}^{R}\right|^{2}\right]\leq\Delta_{i}\mathbb{E}\left[\esssup_{a\in A}
\left|\mathcal{Z}_{i,a}^{R}-\mathcal{P}_{i,a}^{Z}\left(\mathcal{Z}_{i,a}^{R}\right)\right|^{2}\right]+\mathbb{E}\left[\esssup_{a\in A}\left|\Delta\mathcal{Y}_{i+1,a}^{R}\right|^{2}\right]
\]

\end{proof}

\section{Applications\label{sec:Applications}}

In this section, we test our numerical scheme on various examples.

\subsection{Linear Quadratic stochastic control problem}

The first application is an example of a linear-quadratic stochastic
control problem. We consider the following problem:
\begin{align}
v\left(t,x\right) & =\sup_{\alpha\in\mathcal{A}}\mathbb{E}\left[-\lambda_{0}\int_{t}^{T}\left(\alpha_{s}\right)^{2}ds-\lambda_{1}\left(X_{T}^{\alpha}\right)^{2}\right]\label{eq:ValueFct}\\
dX_{s}^{\alpha} & =\left(-\mu_{0}X_{s}^{\alpha}+\mu_{1}\alpha_{s}\right)dt+\left(\sigma_{0}+\sigma_{1}\alpha_{s}\right)dW_{s}\,,\, X_{0}^{\alpha}=0\label{eq:DiffDynamics}
\end{align}
where $\lambda_{i},\mu_{i},\sigma_{i}>0$, $i=1,2$. It is called
linear-quadratic because the drift and the volatility of $X^{\alpha}$
are linear in $\alpha$ and $X^{\alpha}$, while the terms in the
objective function $v$ are quadratic in $\alpha$ and $X^{\alpha}$.
We choose this example as a first, simple application for our numerical
scheme because there exists analytical solutions to this class of
stochastic control problem (cf. \cite{Yong99}) to which our results
can be compared in order to assess the accuracy of our method.

Now, let us look closer to this specific example. As can be seen from
equation \eqref{eq:ValueFct}, the objective function $v$ penalizes
the terminal value $X_{T}^{\alpha}$ of the controlled diffusion if
it is away from zero (with the $-\lambda_{1}\left(X_{T}^{\alpha}\right)^{2}$
term). Hence, $X^{\alpha}$, which starts from zero, has to be controlled
carefully over time so as not divert too much from this initial value.
This can be achieved through the control $\alpha$ in the drift term
($-\mu_{0}X_{s}^{\alpha}+\mu_{1}\alpha_{s}$), which can reinforce
the default mean-reversion speed $\mu_{0}$. However, this control
also impacts the volatility ($\sigma_{0}+\sigma_{1}\alpha_{s}$),
which makes it easier to decrease $X^{\alpha}$ than to increase it.
Moreover, the controls are penalized over time ($-\lambda_{0}\int_{t}^{T}\left(\alpha_{s}\right)^{2}ds$),
meaning that they must be exerted parsimoniously. 

We test our numerical scheme on this specific problem. We set the
parameters to the following values:

\noindent \begin{center}
\begin{tabular}{|c|c|c|c|c|c|c|}
\hline 
$\lambda_{0}$ & $\lambda_{1}$ & $\mu_{0}$ & $\mu_{1}$ & $\sigma_{0}$ & $\sigma_{1}$ & $T$\tabularnewline
\hline 
\hline 
$20$ & $200$ & $0.02$ & $0.5$ & $0.2$ & $0.1$ & $2$\tabularnewline
\hline 
\end{tabular}
\par\end{center}

For the numerical parameters, we use $n=52$ time-discretization steps,
and a sample of $M=10^{6}$ Monte Carlo simulations. For the regressions,
we use a basis function of global polynomial of degree two:
\[
\phi\left(t,x,\alpha\right)=\beta_{0}+\beta_{1}x+\beta_{2}\alpha+\beta_{3}x\alpha+\beta_{4}x^{2}+\beta_{5}\alpha^{2}\,.
\]
In particular, assuming $\beta_{5}<0$, the optimal control will be
linear w.r.t. $x$:
\begin{align*}
\alpha^{*} & =\alpha^{*}\left(t,x\right):=\arg\max_{\alpha}\phi\left(t,x,\alpha\right)=A\left(t\right)x+B\left(t\right)\\
A\left(t\right) & :=-\frac{\beta_{3}}{2\beta_{5}}\,\,\,,\,\,\, B\left(t\right):=-\frac{\beta_{2}}{2\beta_{5}}
\end{align*}
This behaviour is illustrated on Figure \ref{fig:OptimalControl}
below.

\begin{figure}[H]
\hspace{-5.1mm}\subfloat[\label{fig:optimalControlShape}Shape]{\includegraphics[width=0.41\paperwidth]{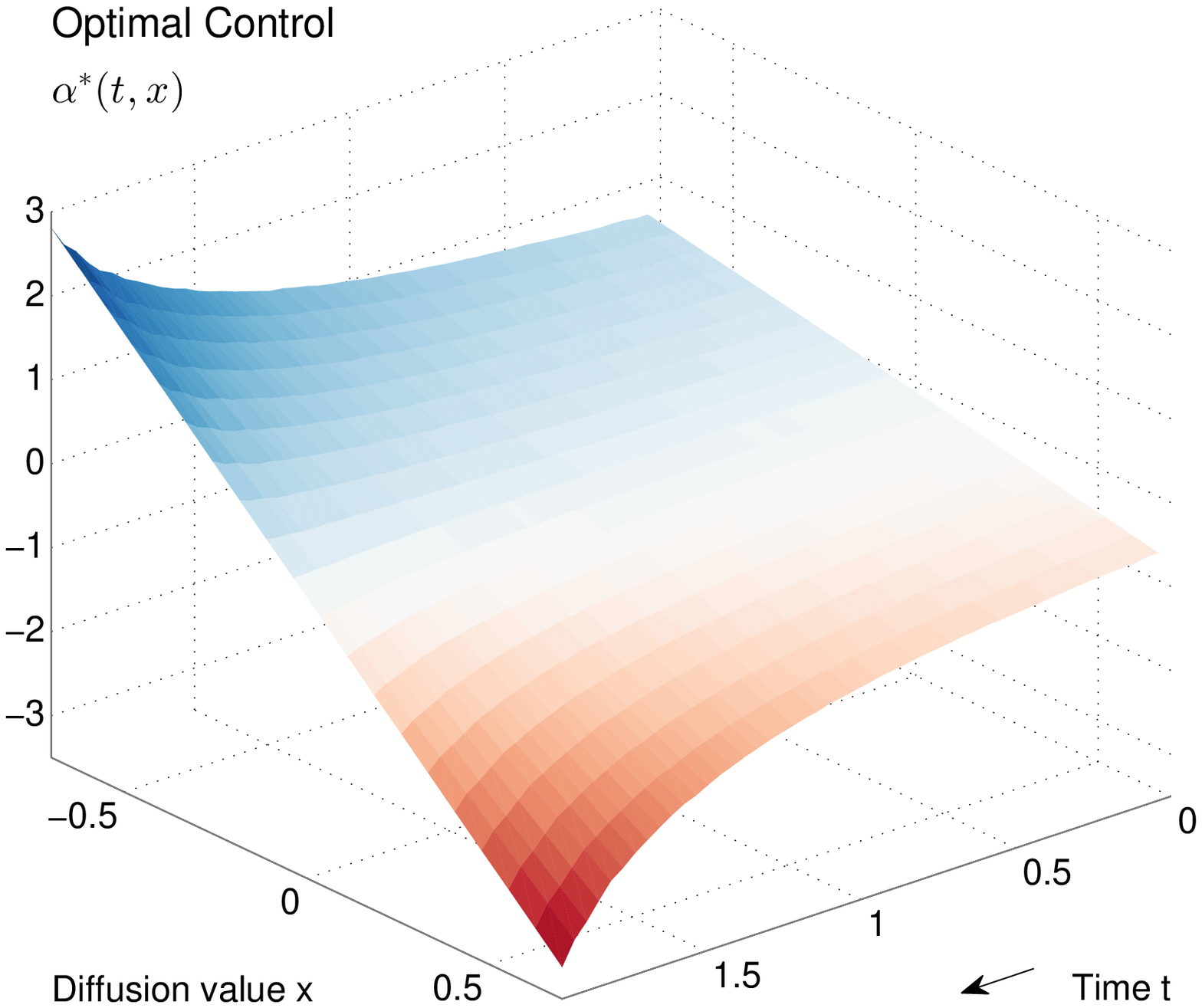}

}\hspace{-3mm}\subfloat[\label{fig:optimalCoefsComparison}Optimal coefficients vs. theoretical
values]{\includegraphics[width=0.41\paperwidth]{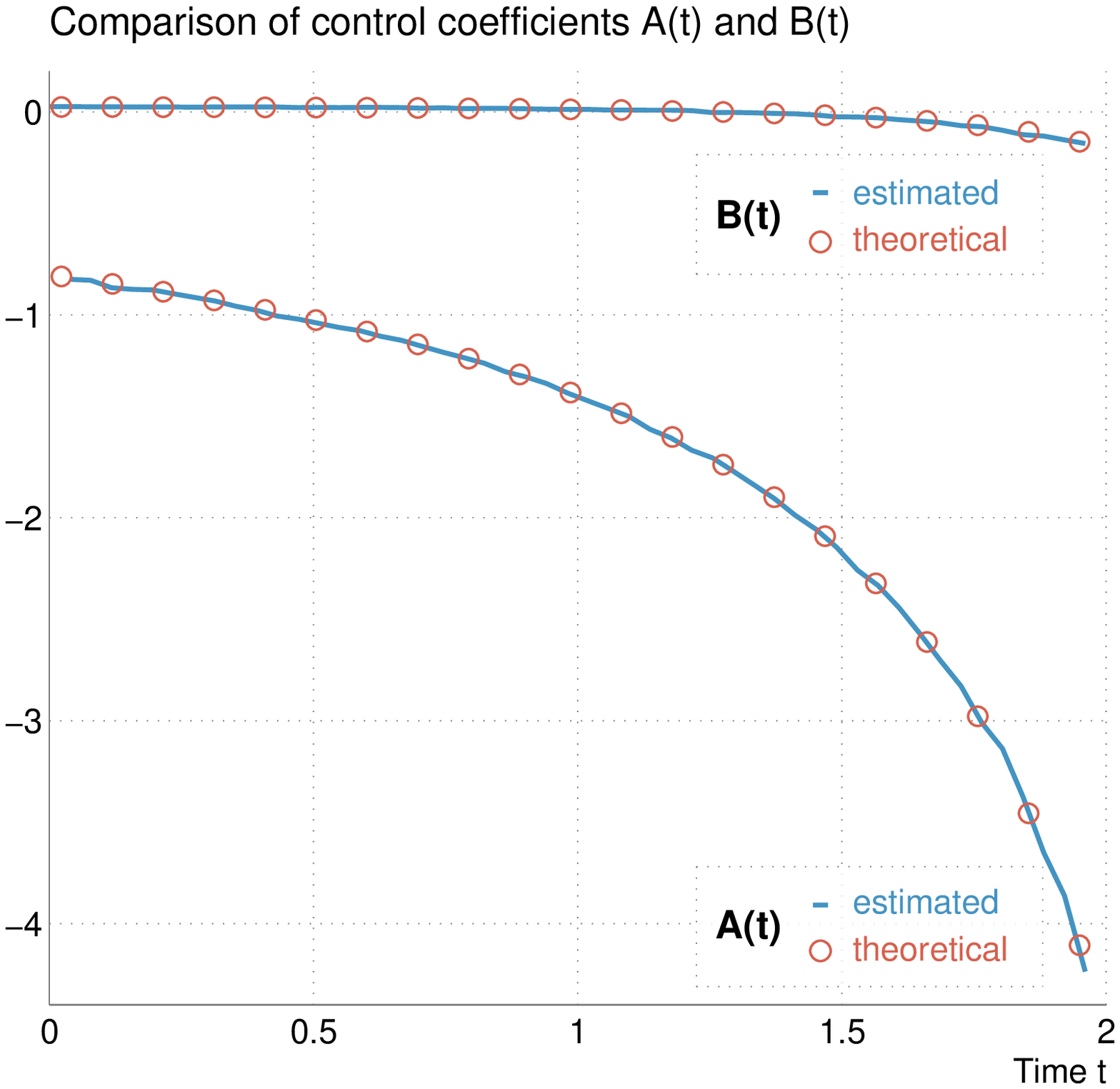}

}\caption{\label{fig:OptimalControl}Optimal control}
\end{figure}

Figure \ref{fig:optimalControlShape} displays the shape of the optimal
control $\alpha^{*}\left(t,x\right)$.

First, as expected from the drift term in the dynamics of $X^{\alpha}$
(equation \eqref{eq:DiffDynamics}), $\alpha^{*}$ is a decreasing
function of $x$ ($A\left(t\right)\leq0$):

- If $X_{t}^{\alpha}$ takes a large positive value, then $\alpha^{*}\left(t,X_{t}\right)$
will take a large negative value so as to push it back more quickly
to zero (recall the drift term $-\mu_{0}X_{s}^{\alpha}+\mu_{1}\alpha_{s}$).

- Conversely, if $X_{t}^{\alpha}$ takes a large negative value, then
$\alpha^{*}\left(t,X_{t}\right)$ will take a large positive value
for the same reason.

Second, the strength of the control increases as time reaches maturity
(i.e. $A\left(t\right)$ decreases with $t$). Indeed, the penalization
of the control becomes relatively cheaper compared with the penalization
of the final value when time is close to maturity.

The strengthening of the control can also be assessed on Figure \ref{fig:optimalCoefsComparison},
which displays the time evolution of the estimated coefficients $A$
and $B$ ($\alpha^{*}\left(t,x\right)=A\left(t\right)x+B\left(t\right)$).
Moreover, one can see that the coefficient $B$ is slightly negative
close to maturity. This creates an asymmetry in the control (as $\alpha^{*}\left(t,0\right)=B\left(t\right)\neq0$),
which comes from the asymmetric effect of the control on the volatility
of $X^{\alpha}$.

The effect of the optimal control $\alpha^{*}$ is clearly visible
on Figure \ref{fig:DiffusionEvolution} below, which compares the
distribution of $X^{\alpha}$ without control (Figure \ref{fig:uncontrolledDiffusion})
and when the optimal control is used (Figure\ref{fig:controlledDiffusion}).
The strengthening of the control at the end of the time period, as
well as the slightly asymmetric shape of the distribution are prominent. 

\begin{figure}[H]
\hspace{-5.1mm}\subfloat[\label{fig:uncontrolledDiffusion}Without control]{\includegraphics[width=0.41\paperwidth]{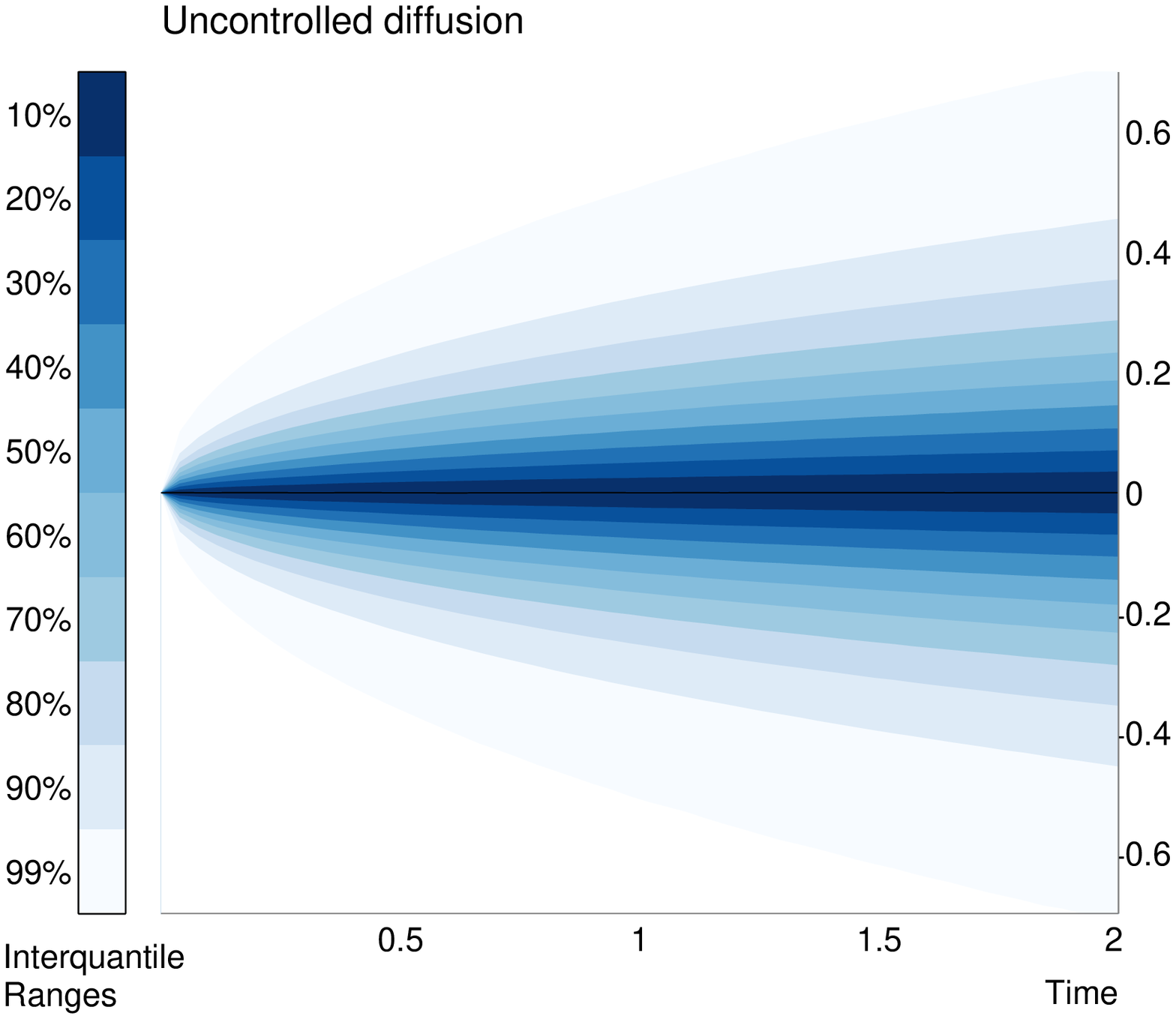}

}\hspace{-3mm}\subfloat[\label{fig:controlledDiffusion}With control]{\includegraphics[width=0.41\paperwidth]{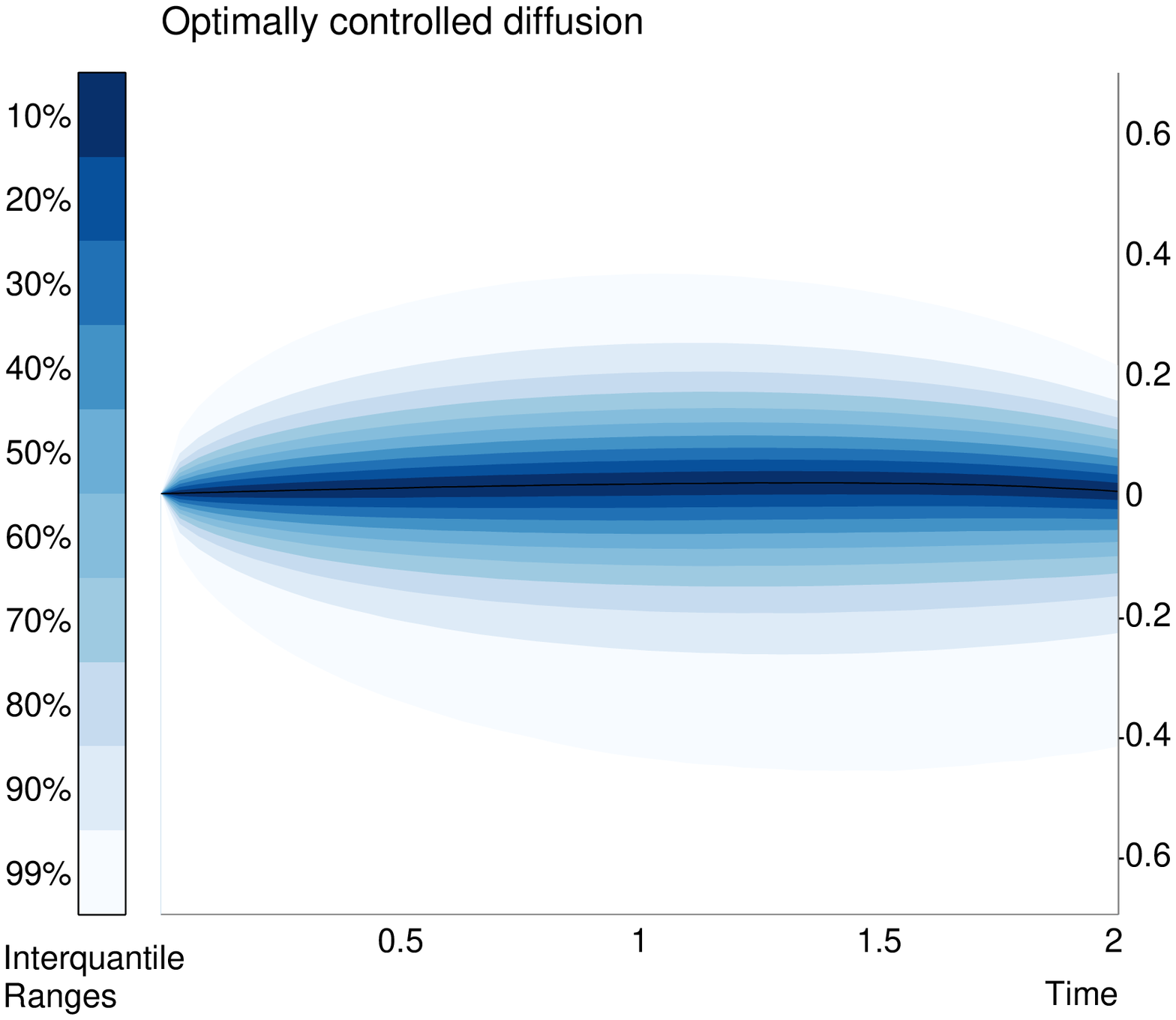}

}\caption{\label{fig:DiffusionEvolution}Time-evolution of the distribution
of the diffusion}
\end{figure}

Finally, regarding the accuracy of the method, the comparison between
the estimated coefficients and their theoretical values is reported
on Figure \ref{fig:optimalCoefsComparison}. Indeed an analytical
characterization of the solution of linear quadratic stochastic control
problems is available using ordinary differential equations (cf. \cite{Yong99}).
On our one-dimensional example \eqref{eq:ValueFct}, it is given by:
\begin{eqnarray*}
\alpha^{*}\left(t,X_{t}\right) & = & A\left(t\right)X_{t}+B\left(t\right)\\
A\left(t\right) & = & -\frac{\mu_{1}P\left(t\right)}{2\lambda_{0}+\sigma_{1}^{2}P\left(t\right)}\\
B\left(t\right) & = & -\frac{\mu_{1}}{2\lambda_{0}}Q\left(t\right)+A\left(t\right)\left(\frac{\sigma_{0}\sigma_{1}}{\mu_{1}}-\frac{\sigma_{1}^{2}}{2\lambda_{0}}Q\left(t\right)\right)
\end{eqnarray*}
where $P\left(t\right)$ and $Q\left(t\right)$ are the solutions
of the following ordinary differential equations:
\begin{eqnarray*}
P'\left(t\right) & = & 2\mu_{0}P\left(t\right)+\frac{\mu_{1}^{2}P^{2}\left(t\right)}{2\lambda_{0}+\sigma_{1}^{2}P\left(t\right)}\\
P\left(T\right) & = & 2\lambda_{1}\\
Q'\left(t\right) & = & \left(\mu_{0}+\frac{\mu_{1}^{2}P\left(t\right)}{2\lambda_{0}+\sigma_{1}^{2}P\left(t\right)}\right)Q\left(t\right)+\frac{\sigma_{0}\sigma_{1}\mu_{1}P^{2}\left(t\right)}{2\lambda_{0}+\sigma_{1}^{2}P\left(t\right)}\\
Q\left(T\right) & = & 0
\end{eqnarray*}
As can be seen from the comparison on Figure \ref{fig:optimalCoefsComparison},
our estimates of the control coefficients are very accurate. Regarding
the value function, our method provides the estimate $\hat{v}\left(0,0\right)=-5.761$.
The theoretical value being equal to $-5.705$, this means a relative
error of $1\%$.

\subsection{Uncertain volatility/correlation model\label{sub:UCM}}

The second application is the problem of pricing and hedging an option
under uncertain volatility.

Instead of specifying the parameters of the dynamics of an underlying
process, one can, for robustness, consider them uncertain. To some
extent, this parameter uncertainty provides hedging strategies that
are more robust to model risk (cf. \cite{Talay08}). To handle these
uncertain parameters, the usual approach is to resort to superhedging
strategies, that is, to find the smallest amount of money from which
it is possible to superreplicate the option, i.e. to build a strategy
that will almost surely provide an amount greater than (or equal to)
the payoff at the maturity of the option.

To compute these prices in practice, the most common approach is to
resort to numerical methods for partial differential equations. For
instance, \cite{Marabel11} computes the superhedging price under
uncertain correlation of a digital outperformance option using a finite
differences sheme. Unfortunately, these PDE methods suffer from the
curse of dimensionality, which means that they cannot handle many
state variables (no more than three in practice).

This is why a few authors tried recently to resort to Monte Carlo
techniques to solve this problem of pricing and hedging options under
uncertain volatility and/or correlation.

To our knowledge, the first attempt to do so was made in \cite{Mrad08}.
In this thesis, along the usual backward induction, the conditional
expectation are computed using the Malliavin calculus approach. This
approach uses the representation of conditional expectations in terms
of a suitable ratio of unconditional expectations. Then, to find the
optimal covariance matrix at each time step, an exhaustive comparison
is performed. Of course, this methodology works only if the set of
possible matrices is finite, which is the case when the optimal control
is of bang-bang type. For instance, it includes the case of unknown
correlations with known volatilies, but not the case when both volatilities
and correlations are unknown, a shortcoming that is acknowledged in
\cite{Mrad08}. This means that this methodology can only deal with
optimal switching problems, for which the control set is finite.

To overcome this limitation, \cite{Guyon11} propose to restrict the
maximization domain to a parameterized set of relevant functions,
indexed by a low-dimensional parameter. They then perform this much
simpler optimization inductively at each time step, by the downhill
simplex method (when the optimum is not of bang-bang type). Once it
is done, say, at time $t_{i}$, they immediately use these estimated
volatilities and correlations (along with those from $t_{j}>t_{i}$)
to resample the whole Monte Carlo set from $t_{i}$ to $T$ (and idea
also used in the Multiple Step Forward scheme from \cite{Gobet11}).
Remark that this parameterization avoids the computation of conditional
expectations for each point and time step.

In \cite{Guyon11}, a second Monte Carlo scheme is proposed. It is
a Monte Carlo scheme for 2-BSDEs, very similar to the schemes \cite{Cheridito07}
and \cite{Fahim11}, but fine-tuned for the uncertain volatility problem
under log-normal processes. The conditional expectations are computed by parametric regression
(non-parametric regression in dimension 1). Then for each point and
each time step, a deterministic optimization procedure has to be performed
to find the optimal covariance matrix. However, unlike in their previous
algorithm, there is no resampling of the underlying diffusion using
the newly computed covariances, which means that ensuring a proper
simulation of the forward process becomes an issue.

Finally, we would like to draw attention to the work \cite{NguyenHuu13},
which is not devoted to the uncertain volatility problem (it deals
with the partial hedging of power futures with others futures with
larger delivery period), but the numerical scheme they propose can
deal with a control in the volatility. Their specific application
allows to retrieve the optimal control by a fixed point argument,
within a backward scheme. However, as in the previous algorithm, an
a priori control has to be used to simulate the forward process.

In the present paper, our numerical scheme provides an alternative
numerical sheme for dealing with the problem of pricing and hedging
an option under uncertain volatility. To illustrate this, we implement
it below on a simple example.

Consider two underlyings driven by the following dynamics:
\begin{eqnarray}
dS_{i}\left(t\right) & = & \sigma_{i}S_{i}\left(t\right)dW_{i}\left(t\right)\,\,,\,\, i=1,2\label{eq:uvmS}\\
\left\langle dW_{1}\left(t\right),dW_{2}\left(t\right)\right\rangle  & = & \rho\left(t,S_{1}\left(t\right),S_{2}\left(t\right)\right)dt\label{eq:uvmRho}
\end{eqnarray}
where $\sigma_{1},\sigma_{2}>0$, $W_{1}$ and $W_{2}$ are two correlated
brownian motions. We consider no drift and no interest rate for simplicity.
We instead focus our attention on the following crucial feature: we
consider the correlation $\rho$ to be \textit{uncertain}. We only
assume that $\rho$ always lies between two known bounds $-1\leq\rho_{\min}\leq\rho_{\max}\leq1$:
\begin{equation}
\rho_{\min}\leq\rho\leq\rho_{\max}\label{eq:rhoBounds}
\end{equation}
Notice that when $\rho_{\min}$ $=$ $-1$ or $\rho_{\max}$ $=$ $1$, the diffusion matrix of  $(S_1,S_2)$ can be degenerate.

We could also consider the two volatilities to be uncertain as well,
but for illustration purposes, we focus on the uncertainty of the
correlation parameter.

Finally, consider a payoff function $\Phi=\Phi\left(T,S_{1}\left(T\right),S_{2}\left(T\right)\right)$
at a time horizon $T>0$.

Now, the problem is to estimate the price of an option that delivers
the payoff $\Phi$ at time $T$, and, if possible, to build a hedging
strategy for this option.

Given that $\rho$ is uncertain, the model is incomplete, i.e. it
is not possible to construct a hedging strategy that replicates perfectly
the payoff $\Phi$ from any given amount of money. We thus look for
superhedging strategies instead.

Hence, consider the class $\mathbf{Q}$ of all probability measures
$\mathbb{Q}$ on the sets of paths $\left\{ S_{i}\left(t\right)\right\} _{0\leq t\leq T}^{i=1,2}$
such that equations \eqref{eq:uvmRho} and \eqref{eq:rhoBounds} hold
for a particular $\rho^{\mathbb{Q}}$. The superhedging price is thus
given by:
\begin{equation}
P_{0}^{+}:=\sup_{\mathbb{Q}\in\mathbf{Q}}\mathbb{E}^{\mathbb{Q}}\left[\Phi\left(T,S_{1}\left(T\right),S_{2}\left(T\right)\right)\right]\label{eq:superreplicationPrice}
\end{equation}
and the superhedging strategy is simply given by the usual delta-hedging
strategy with $\rho$ equal to the correlation that attains the supremum
in equation \eqref{eq:superreplicationPrice}. In particular it provides
an upper arbitrage bound to the price of the option. Symmetrically,
a lower bound is provided by the subreplication price:
\begin{equation}
P_{0}^{-}:=\inf_{\mathbb{Q}\in\mathbf{Q}}\mathbb{E}^{\mathbb{Q}}\left[\Phi\left(T,S_{1}\left(T\right),S_{2}\left(T\right)\right)\right]\label{eq:subreplicationPrice}
\end{equation}
The practical computation of $P_{0}^{+}$ and $P_{0}^{-}$ falls within
the scope of our numerical scheme.

We thus test our numerical scheme on this specific problem. We consider
the example of a call spread on the spread $S_{1}\left(T\right)-S_{2}\left(T\right)$,
i.e.:
\[
\Phi=\left(S_{1}\left(T\right)-S_{2}\left(T\right)-K_{1}\right)^{+}-\left(S_{1}\left(T\right)-S_{2}\left(T\right)-K_{2}\right)^{+}
\]
 where $K_{1}<K_{2}$. Unless stated otherwise, the parameters of
the model are fixed to the following values:

\noindent \begin{center}
\begin{tabular}{|c|c|c|c|c|c|c|c|c|}
\hline 
$S_{1}\left(0\right)$ & $S_{2}\left(0\right)$ & $\sigma_{1}$ & $\sigma_{2}$ & $\rho_{\min}$ & $\rho_{\max}$ & $K_{1}$ & $K_{2}$ & $T$\tabularnewline
\hline 
\hline 
$50$ & $50$ & $0.4$ & $0.3$ & $-0.8$ & $0.8$ & $-5$ & $5$ & $0.25$\tabularnewline
\hline 
\end{tabular}
\par\end{center}

For the numerical parameters, we use $n=26$ time-discretization steps,
and a sample of $M=10^{6}$ Monte Carlo simulations. For the regressions,
we use a basis function of sigmoid transforms of polynomial of degree
two:
\begin{eqnarray*}
\phi\left(t,s_{1},s_{2},\rho\right) & := & \left(K_{2}-K_{1}\right)\times\mathcal{S}\left(\beta_{0}+\beta_{1}s_{1}+\beta_{2}s_{2}+\beta_{3}\rho+\beta_{4}\rho s_{1}+\beta_{5}\rho s_{2}\right)\\
\mathcal{S}\left(u\right) & := & \frac{1}{1+e^{-u}}
\end{eqnarray*}
We chose the sigmoid function for its resemblance to the call spread
payoff, and the terms inside the sigmoid according to their statistical
significance. With this choice of basis, the optimal control will
be bang-bang:
\begin{align*}
\rho^{*} & =\rho^{*}\left(t,s_{1},s_{2}\right):=\arg\max_{\rho}\phi\left(t,s_{1},s_{2},\rho\right)=\rho_{\max}\mathbf{1}\left\{ \beta_{3}+\beta_{4}s_{1}+\beta_{5}s_{2}\geq0\right\} +\rho_{\min}\mathbf{1}\left\{ \beta_{3}+\beta_{4}s_{1}+\beta_{5}s_{2}<0\right\} 
\end{align*}

Figure \ref{fig:UCM} below reports our results.

Figure \ref{fig:CallSpread} reports the superhedging and subhedging
prices of the option, for different values of the moneyness ($S_{2}\left(0\right)=50$
is kept fixed and different values of $S_{1}\left(0\right)=50+{\rm Moneyness}$
are tested). One can clearly see the range of non-arbitrage prices
that they define. For comparison, the prices obtained when $\rho$
is constant are reported on the same graph for different values ($\rho_{\min}$,
$0$ and $\rho_{\max}$). One can see that, even though these prices
belong to the non-arbitrage range, they do not cover the whole range,
especially close to the money. This clearly indicates that, as already
observed in \cite{Marabel11} for instance, the practice of pricing
under the hypothesis of constant parameters, and then testing different
values for the parameters can be a very deceptive assessment of risk
(as ``uncertain'' is not the same as ``uncertain but constant'').

Figure \ref{fig:CorrelRange} illustrates the impact of the size of
the correlation range $\left[\rho_{\min},\rho_{\max}\right]$. Naturally,
the wider the correlation range, the wider the price range. On average,
an increase of 0.1 of the correlation range increases the price range
by $0.135$.

\begin{figure}[H]
\hspace{-5.1mm}\subfloat[\label{fig:CallSpread}Price of Call Spread]{\includegraphics[width=0.41\paperwidth]{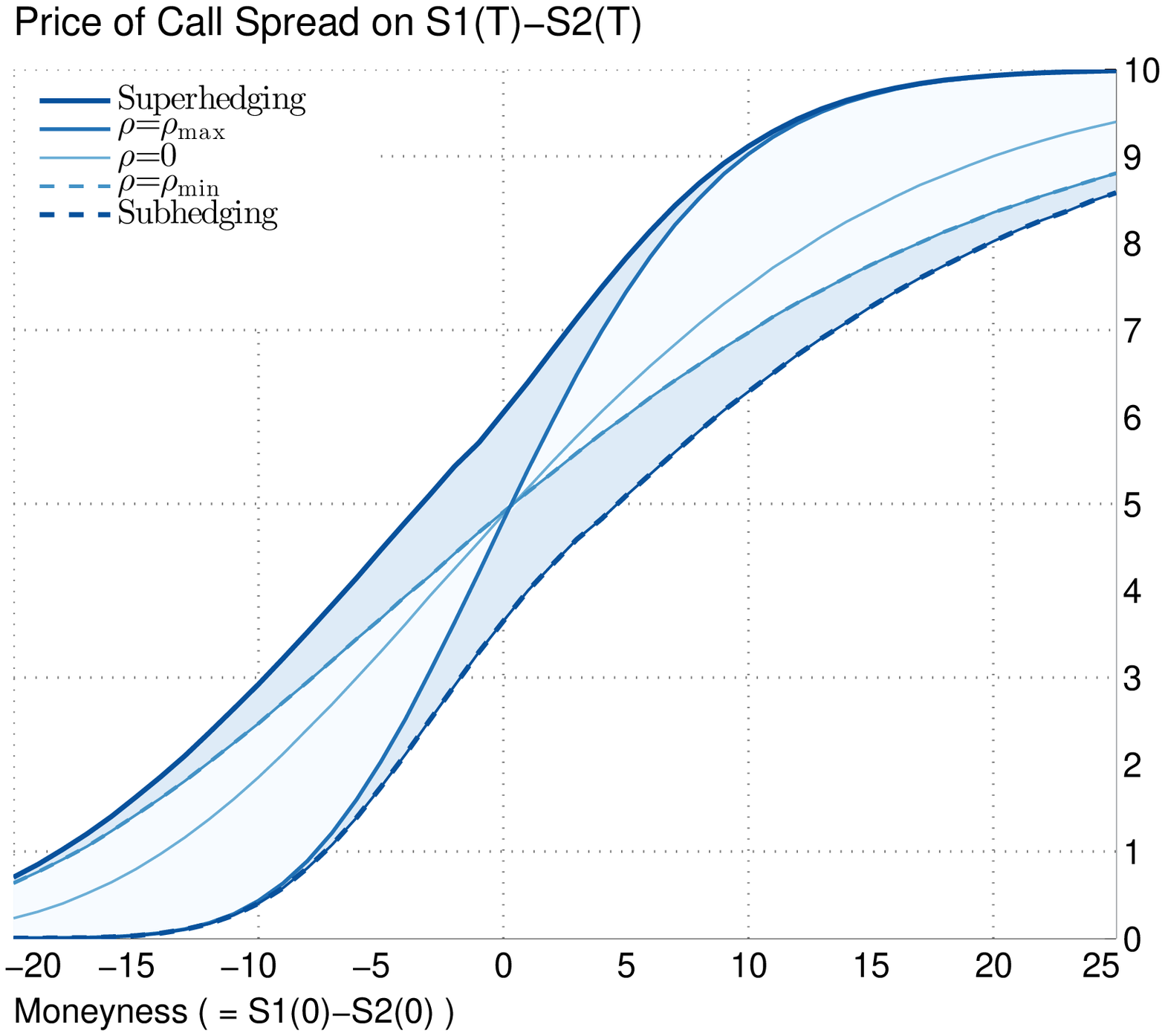}

}\hspace{-3mm}\subfloat[\label{fig:CorrelRange}Influence of the correlation range]{\includegraphics[width=0.41\paperwidth]{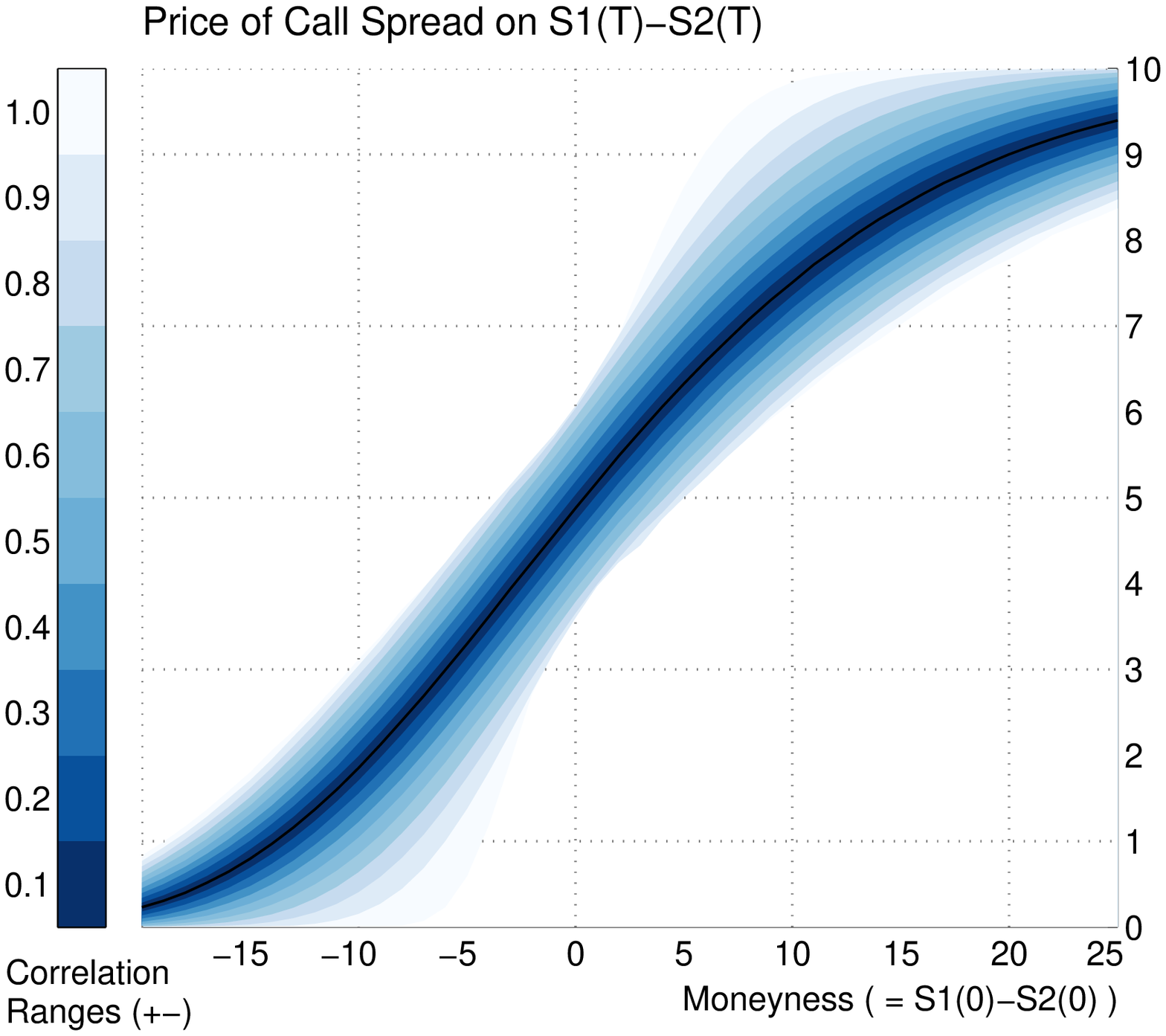}

}\caption{\label{fig:UCM}Prices under uncertain correlation}
\end{figure}

\subsection{Comparisons with \cite{Guyon11}}

Finally, we test our algorithm on several payoffs proposed in \cite{Guyon11},
and compare the behaviour of our method to their results. To be more
specific, we will not focus our comparison of algorithms to their
parametric approach%
\footnote{For comprehensiveness, here are the main pros and cons of the parametric
approach: it is very accurate (especially when the optimal control
belongs to the chosen parametric class) but requires $\mathcal{O}\left(N^{2}\times M\right)$
operations, as at each time step $t_{i}$ the simulations of the forward
process are recomputed between $t_{i}$ and $t_{N}$ using the newly
estimated optimal controls. %
}, but to their second-order BSDE approach, as both algorithms are
similar in nature (forward-backward schemes involving simulations
and regressions). 

Actually, we are going to implement and compare two different versions
of our scheme. The first one correspond to the empirical version of
the scheme studied in Section \ref{sec:Scheme}:\foreignlanguage{french}{{\small 
\begin{align}
\hat{Y}_{N} & =g\left(X_{N}\right)\nonumber \\
\hat{\mathcal{Y}}_{i} & =\hat{\mathbb{E}}_{i}\left[\hat{Y}_{i+1}+f\left({\color{black}X_{i}},I_{i}\right)\Delta_{i}\right]\nonumber \\
\hat{Y}_{i} & =\esssup_{a\in A}\mathbb{E}_{i,a}\left[\hat{\mathcal{Y}}_{i}\right]\label{eq:TvR}
\end{align}
}}{\small \par}
where $\hat{\mathbb{E}}_{i}$ corresponds to an empirical least-squares
regression which approximates the true conditional expectation $\mathbb{E}_{i}$.
In the simpler context of American option pricing, this scheme would
correspond to the Tsitsiklis-van Roy algorithm (\cite{Tsitsiklis01}).

The second one makes use of the estimated optimal policies computed
by the first algorithm, which are then directly plugged into the stochastic
control problem under consideration:\foreignlanguage{french}{{\small 
\begin{align}
\hat{\alpha}_{i} & =\arg\esssup_{a\in A}\mathbb{E}_{i,a}\left[\hat{\mathcal{Y}}_{i}\right]\nonumber \\
\hat{X}_{i+1} & =b(\hat{X}_{i},\hat{\alpha}_{i})\Delta_{i}+\sigma(\hat{X}_{i},\hat{\alpha}_{i})\Delta W_{i}\nonumber \\
\hat{v}\left(t_{0},x_{0}\right) & =\frac{1}{M}\sum_{m=1}^{M}\left[\sum_{i=1}^{N}f(\hat{X}_{i+1},\hat{\alpha}_{i})\Delta_{i}+g(\hat{X}_{N})\right]\label{eq:LS}
\end{align}
}}{\small \par}
In the context of American option pricing, this scheme would correspond
to the Longstaff-Schwarz algorithm (\cite{Longstaff01}).

We compute both prices as they are somehow complementary. Indeed,
as noticed in \cite{Bouchard11} and detailed in \cite{Aid12-2},
the first algorithm tend to be upward biased (up to the Monte Carlo
error and the regression bias) compared with the discretized price,
while the second one tend to be downward biased (up to the Monte Carlo
error). Therefore, computing both prices provides a kind of empirical
confidence interval, with the length of the interval being due to
the choice of regression basis, thus providing an empirical assessment
of the quality of the chosen regression basis.

\paragraph*{Call Spread}

Let $S$ be a geometric brownian motion with $S\left(0\right)=100$
and with uncertain volatility $\sigma$ taking values in $\left[0.1,0.2\right]$.

Consider a call spread option, with payoff $\left(S\left(T\right)-K_{1}\right)^{+}-\left(S\left(T\right)-K_{2}\right)^{+}$
and time horizon $T=1$, with $K1=90$ and $K2=110$. The true price
of the option (as estimated by PDE methods in \cite{Guyon11}) is
$\mathcal{C}_{PDE}=11.20$, and the Black-Scholes price with constant
volatility $\sigma_{\mathrm{mid}}=0.15$ is $\mathcal{C}_{BS}=9.52$.
We implement our scheme using the following set of basis functions:
\[
\phi\left(t,s,\sigma\right)=\left(K_{2}-K_{1}\right)\times\mathcal{S}\left(\beta_{0}+\beta_{1}s+\beta_{2}s^{2}+\beta_{3}\sigma+\beta_{4}\sigma s+\beta_{5}\sigma s^{2}\right)
\]
where, as in Subsection \ref{sub:UCM}, $\mathcal{S}$ denotes the
sigmoid function. 

Figure \ref{fig:UVM-CS-1D} describes the estimates obtained with
both algorithms \eqref{eq:TvR} and \eqref{eq:LS}, for various values
of the number $M$ of Monte Carlo simulations, and of the length of
the constant discretization time step. For comparison, the red line
corresponds to the price $\mathcal{C}_{PDE}$ of the option.

\begin{figure}[H]
\hspace{-4mm}\subfloat[\label{fig:CallSpread1}First Algorithm]{\includegraphics[width=0.4\paperwidth]{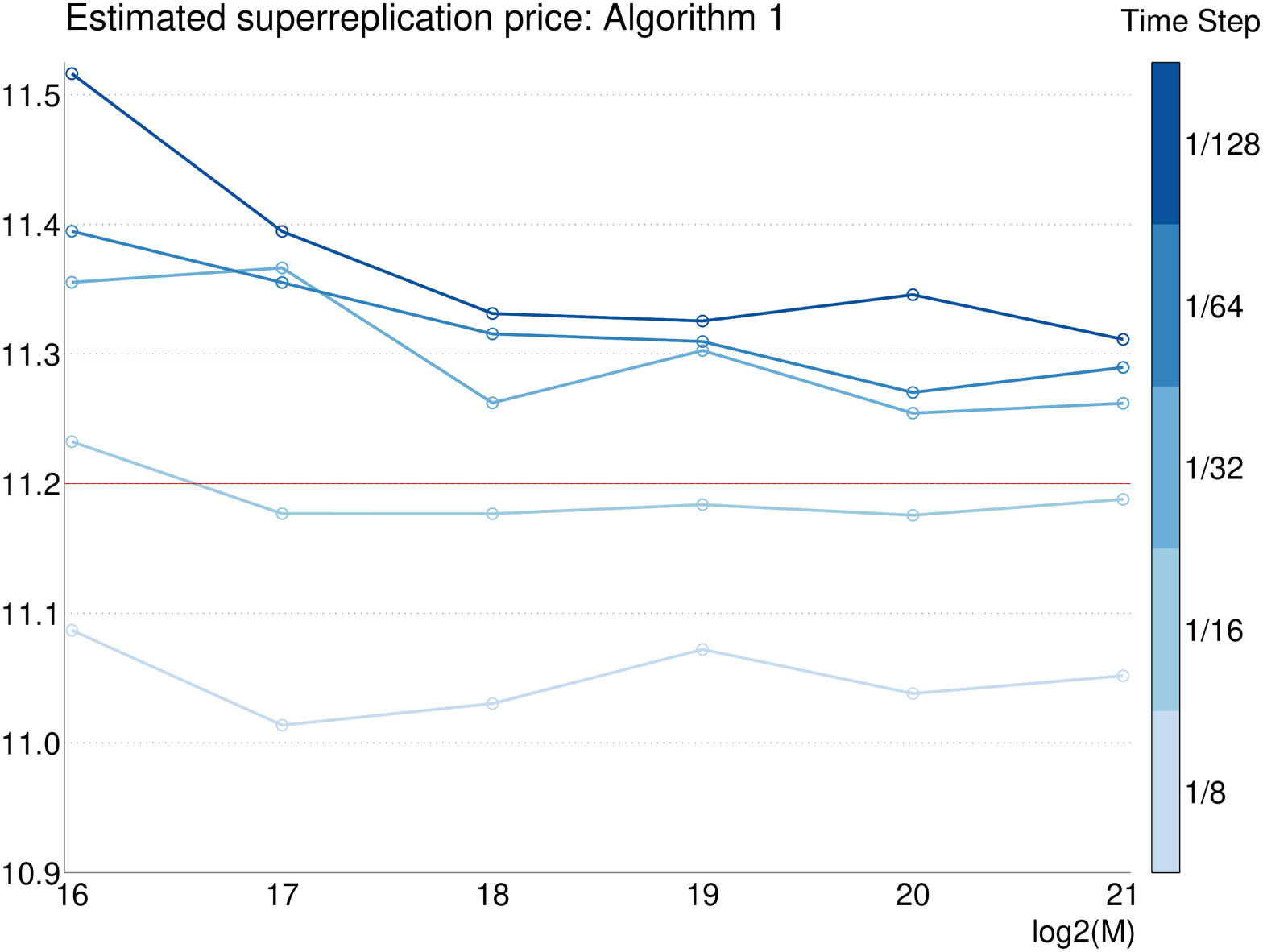}

}\subfloat[\label{fig:CallSpread2}Second Algorithm]{\includegraphics[width=0.4\paperwidth]{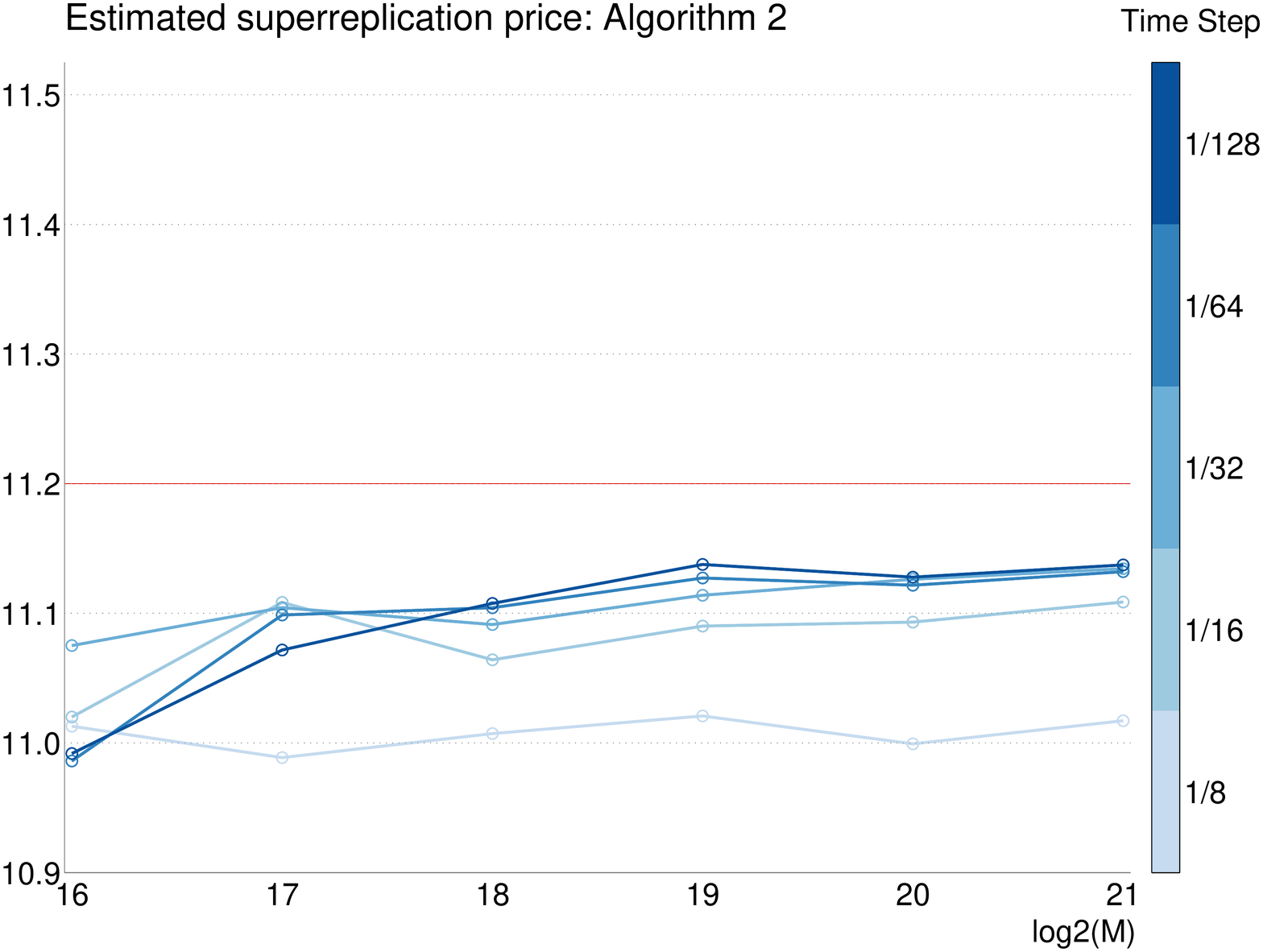}

}\caption{\label{fig:UVM-CS-1D}Price of Call Spread}
\end{figure}

The following general observations can be made.

First, for a small enough time step, the prices computed using the
first algorithm \eqref{eq:TvR} (Figure \ref{fig:CallSpread1}) tend
as expected to be above the true price, while the second algorithm
\eqref{eq:LS} (Figure \ref{fig:CallSpread2}) tend to be below it.

Our best estimate here ($M=2^{21}$ , $\Delta_{t}=1/128$) is $11.31$
with the first algorithm ($+1\%$ compared with the true price) and $11.14$
with the second one ($-0.6\%$). The true price lies indeed between
those two bounds, and their average ($11.22$) is even closer to the
true price than any of the two estimates ($+0.2\%$).

The prices computed with the first algorithm always lie above the
prices computed with the second algorithm. As these prices are expected
to surround the true discretized price (as would be computed by the
scheme \eqref{eq:TvR} with $\mathbb{E}_{i}$ instead of $\hat{\mathbb{E}}_{i}$),
the fact that for large discretization steps ($\Delta_{t}=1/8$ or
$1/16$) the prices computed using the first algorithm are below the
true price $11.20$ simply means that, for such discretization steps,
the true discretized price lies below the true price (in other words
the time discretization generates here a negative bias).

Finally, increasing the number of Monte Carlo simulations tends as
expected to improve the price estimates. However, the Monte Carlo
error can be negligible compared with the discretization error for small
time steps, which is why both a large number of Monte Carlo simulations
and a small discretization time step are required to obtain accurate
estimates.

In \cite{Guyon11}, the algorithm based on second-order BSDEs produces
the estimates $11.04$ for $\left(1/\Delta_{t},\log_{2}\left(M\right)\right)=\left(8,16\right)$
and $11.11$ for $\left(1/\Delta_{t},\log_{2}\left(M\right)\right)=\left(8,17\right)$.
This is close to our estimates for similar parameters. However, a
more accurate comparison would require to test their algorithm with
smaller time steps and more Monte Carlo simulations (they only consider
parameters $\left(1/\Delta_{t},\log_{2}\left(M\right)\right)$ within
$\left[2,8\right]\times\left[12,17\right]$, whereas we consider here
the range $\left[8,128\right]\times\left[16,21\right]$, as it provides
much greater accuracy of the estimates, providing a sound basis for
the analysis of the results).

\paragraph*{Digital option:}

Consider a digital option, with payoff $100\times\mathbf{1}\left\{ S\left(T\right)\geq K\right\} $
and $T=1$ on the samee asset, with $K=100$. The true (PDE) price
is $\mathcal{C}_{PDE}=63.33$, and the Black-Scholes price with mid-volatility
is $\mathcal{C}_{BS}=46.54$. We use the following set of basis functions:
\[
\phi\left(t,s,\sigma\right)=100\times\mathcal{S}\left(\beta_{0}+\beta_{1}s+\beta_{2}s^{2}+\beta_{3}\sigma+\beta_{4}\sigma s+\beta_{5}\sigma s^{2}\right)
\]

\begin{figure}[H]
\hspace{-4mm}\subfloat[\label{fig:Digital1}First Algorithm]{\includegraphics[width=0.4\paperwidth]{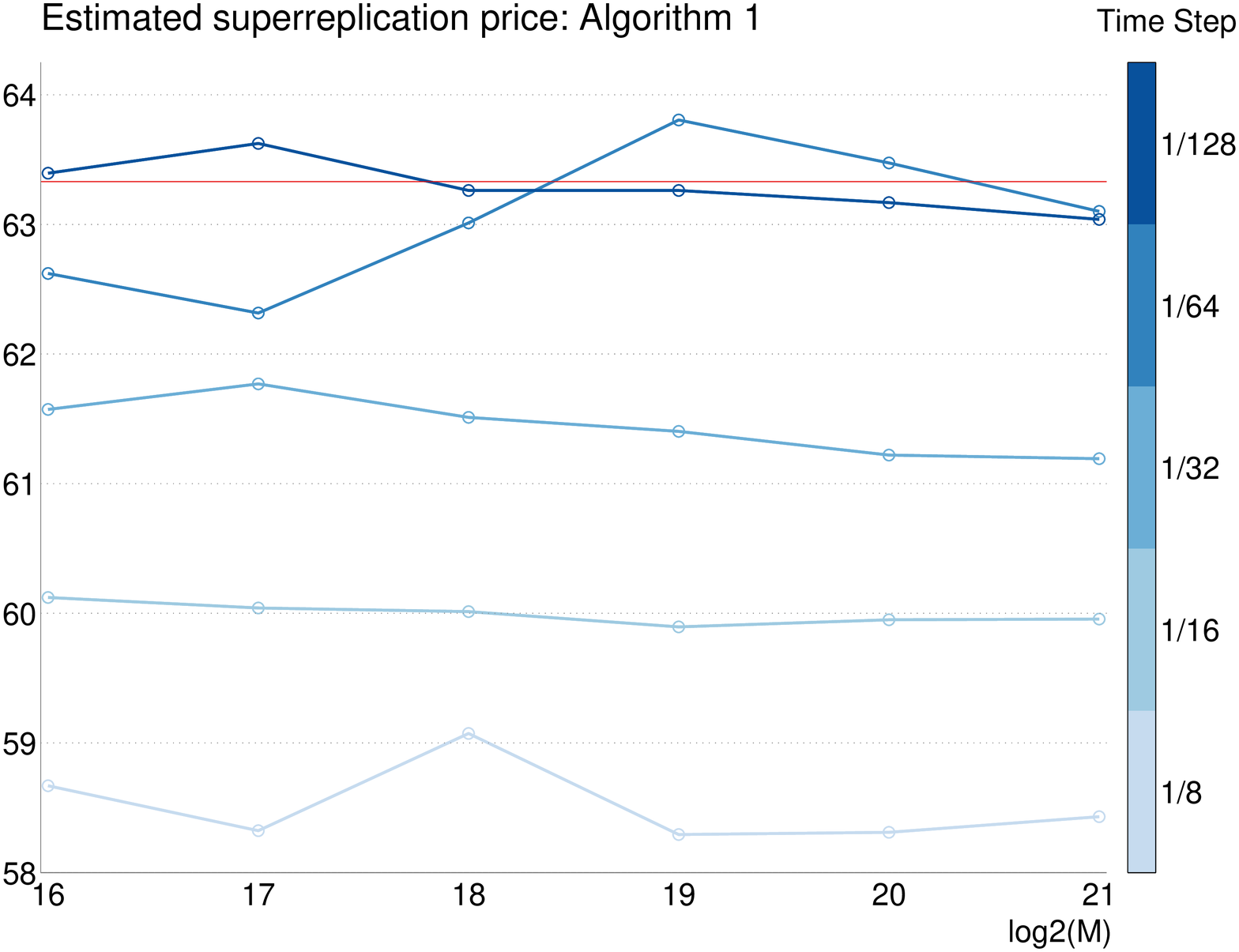}

}\subfloat[\label{fig:Digital2}Second Algorithm]{\includegraphics[width=0.4\paperwidth]{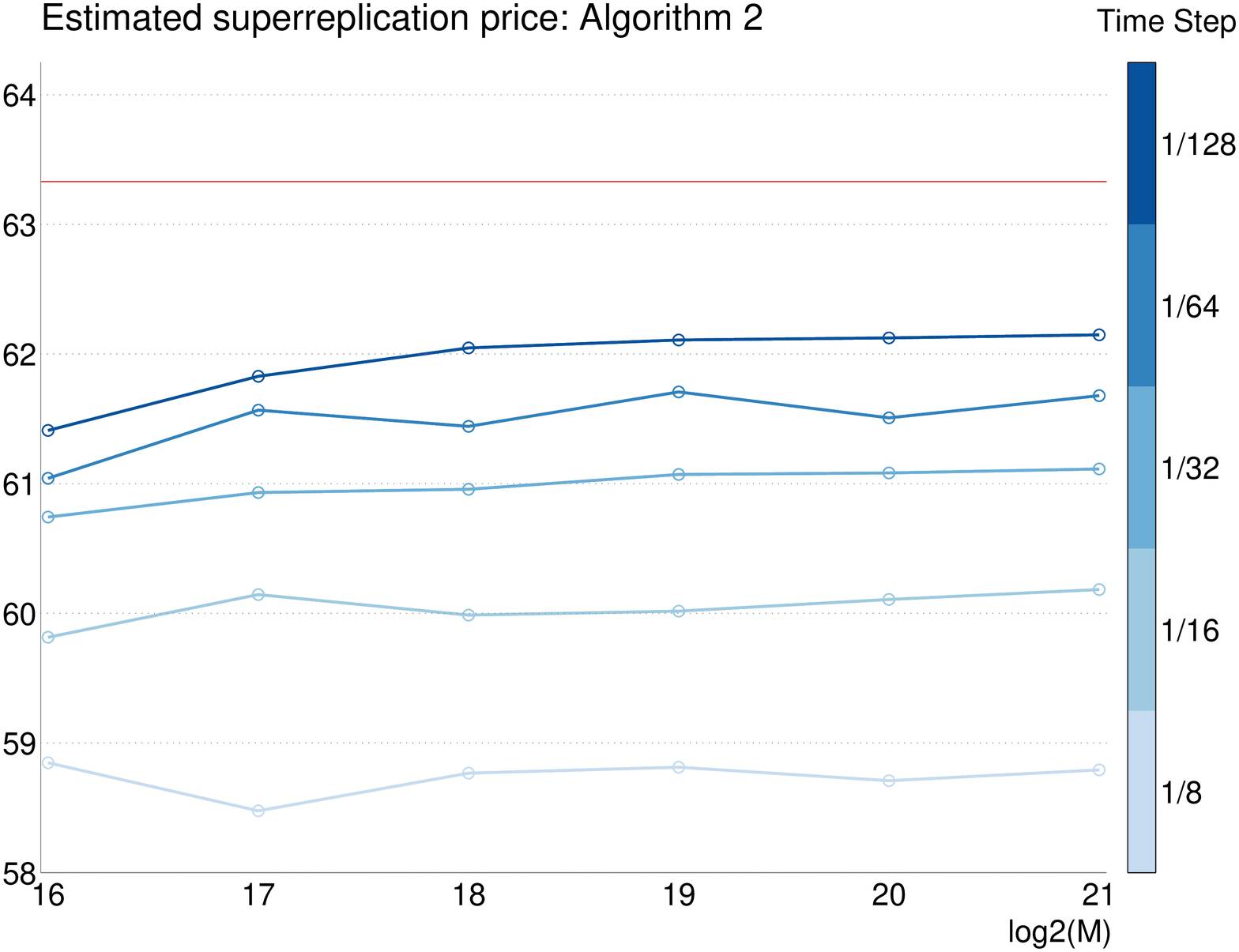}

}\caption{\label{fig:UVM-DIG-1D}Price of Digital Option}
\end{figure}

As can be seen on Figure \ref{fig:UVM-DIG-1D}, the time discretization
error is much more pronounced with this discontinuous payoff, compared
with the previous call spread example. We manage to reach estimates
of $63.04$ ($-0.5\%$) and $62.15$ ($-1.9\%$), even though smaller
time steps would be required for better accuracy.

For small parameters ($\left(1/\Delta_{t},\log_{2}\left(M\right)\right)=\left(8,16\right)$),
the accuracy is better in \cite{Guyon11} ($60.53$), even though
shortening the time step tends to degrade the results in their case.

\paragraph*{Outperformer Option:}

Consider now two geometric Brownian motions $S_{1}$ and $S_{2}$,
starting from $100$ at time $0$, with uncertain volatilities $\sigma_{1}$
and $\sigma_{2}$ taking values in $\left[0.1,0.2\right]$. For the
moment, suppose that the correlation $\rho$ between the two underlying
Brownian motions is zero.

Consider an outperformer option, with payoff $\left(S_{1}\left(T\right)-S_{2}\left(T\right)\right)^{+}$
and time horizon $T=1$. The true price is $\mathcal{C}=11.25$. We
use the following set of basis functions:
\begin{align*}
\phi\left(t,s_{1},s_{2},\sigma_{1},\sigma_{2}\right)= & 100\times\left(\beta_{0}+\beta_{1}s_{1}+\beta_{2}s_{1}^{2}+\beta_{3}s_{2}+\beta_{4}s_{2}^{2}+\beta_{5}s_{1}s_{2}+\beta_{6}\sigma_{1}+\beta_{7}\sigma_{1}s_{1}+\beta_{8}\sigma_{1}s_{1}^{2}\right.\\
 & \left.+\beta_{9}\sigma_{1}s_{2}+\beta_{10}\sigma_{1}s_{2}^{2}+\beta_{11}\sigma_{2}+\beta_{12}\sigma_{2}s_{1}+\beta_{13}\sigma_{2}s_{1}^{2}++\beta_{14}\sigma_{2}s_{2}+\beta_{15}\sigma_{2}s_{2}^{2}\right)
\end{align*}

\begin{figure}[H]
\hspace{-4mm}\subfloat[\label{fig:Outperformer1}First Algorithm]{\includegraphics[width=0.4\paperwidth]{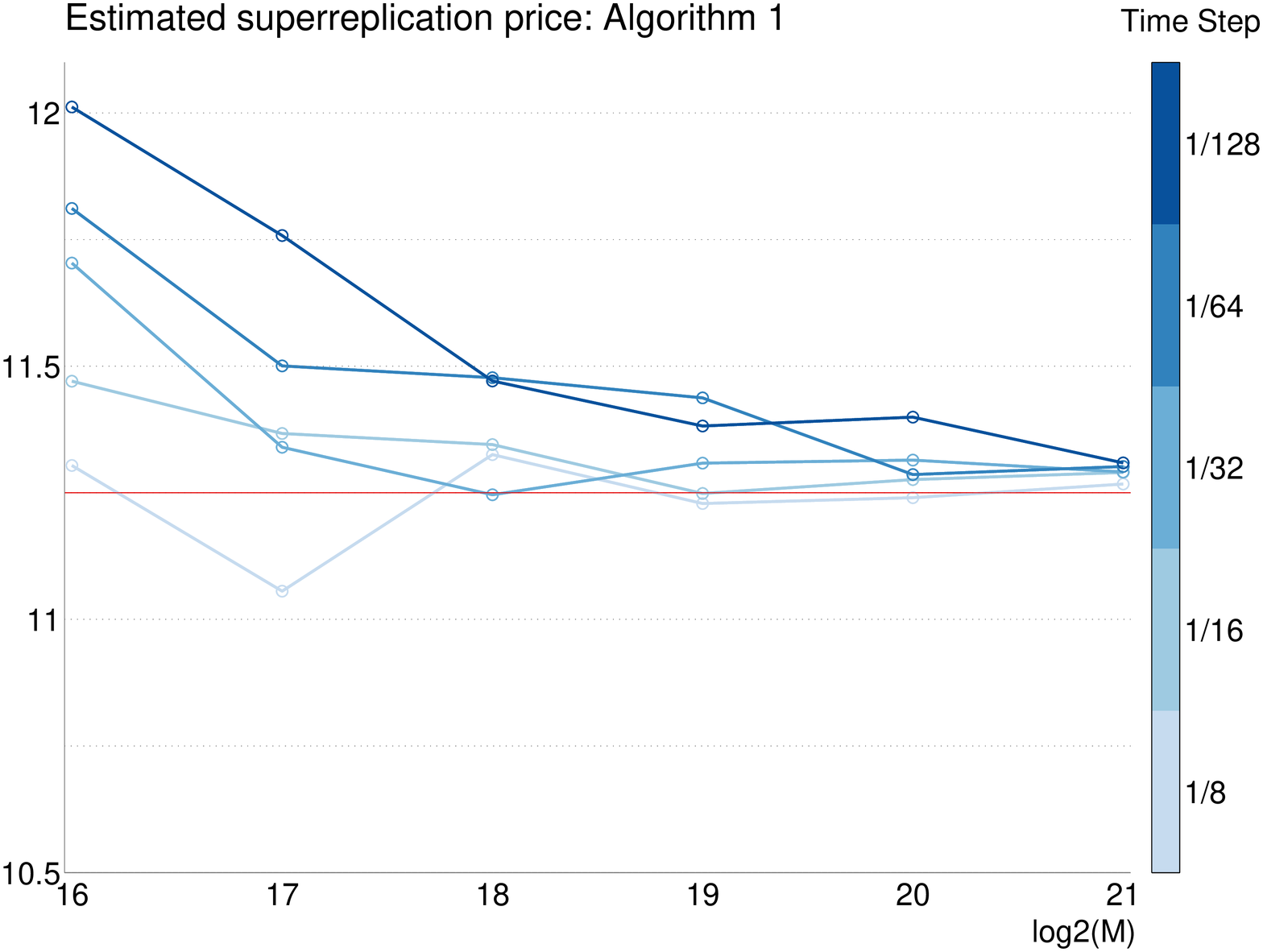}

}\subfloat[\label{fig:Outperformer2}Second Algorithm]{\includegraphics[width=0.4\paperwidth]{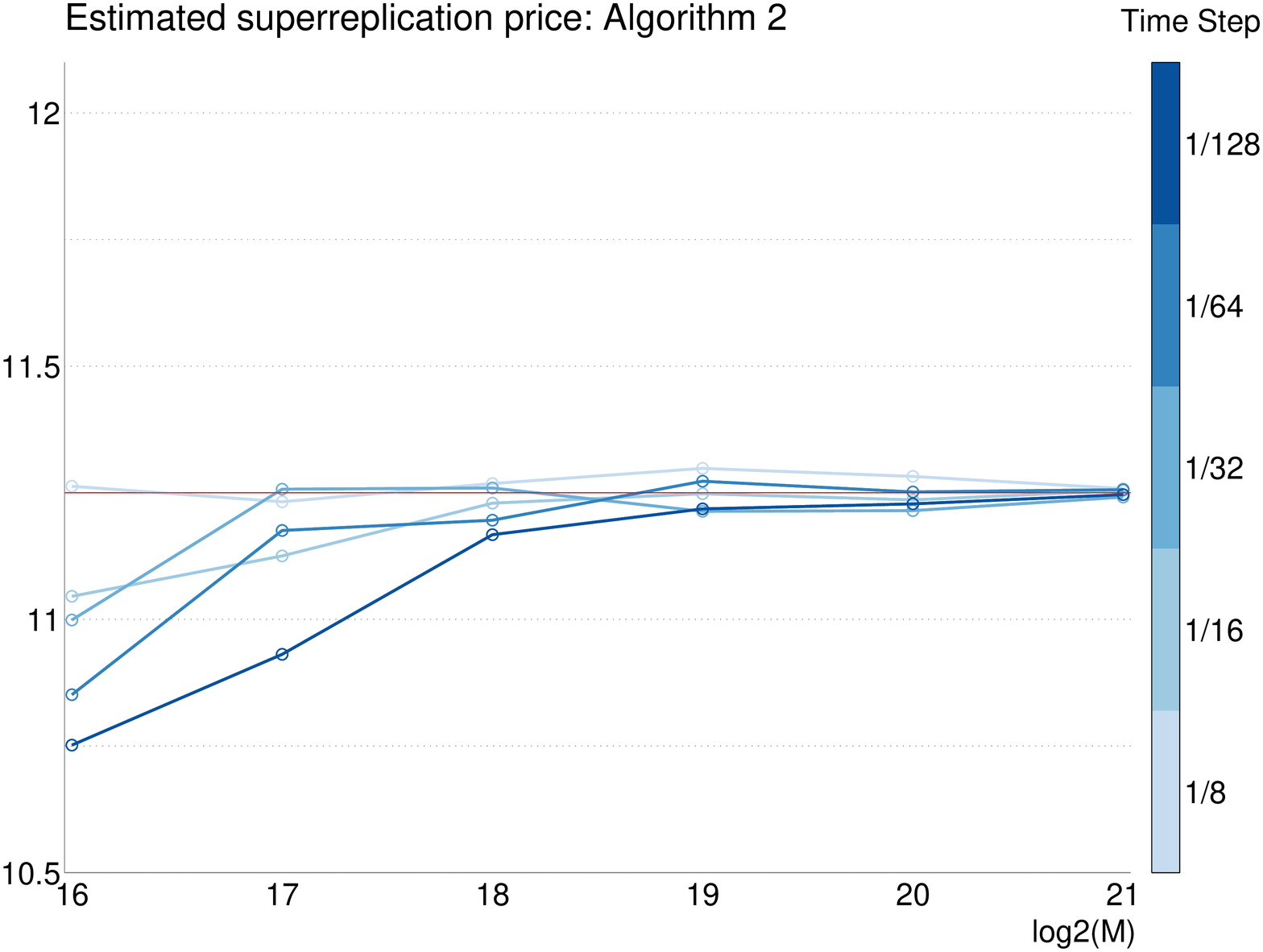}

}\caption{\label{fig:UVM-OPF-2D}Price of Outperformer Option ($\rho=0$)}
\end{figure}

Here, in contrast with the previous examples, the bulk of the error
comes from the Monte Carlo simulations, and not from the time discretization.
Moreover, both algorithms provide very accurate estimates. Indeed,
this convex option is easy to price under the uncertain volatility
model, as it is given by the price obtained with the maximum volatilities.
With our choice of regression basis, the algorithm correctly detects
that the maximum volatilities are to be used, leading to these very
accurate estimates $11.31$ ($+0.5\%$) and $11.25$ ($-0\%$). For
the same reason, the estimates from \cite{Guyon11} are accurate too.

Figure \ref{fig:UVM-OPF-2D-RHO} below depicts the estimated price
of the same option but now with a negative constant correlation $\rho=-0.5$.
Its true price is $\mathcal{C}=13.75$.

\begin{figure}[H]
\hspace{-4mm}\subfloat[\label{fig:OutperformerCorrel1}First Algorithm]{\includegraphics[width=0.4\paperwidth]{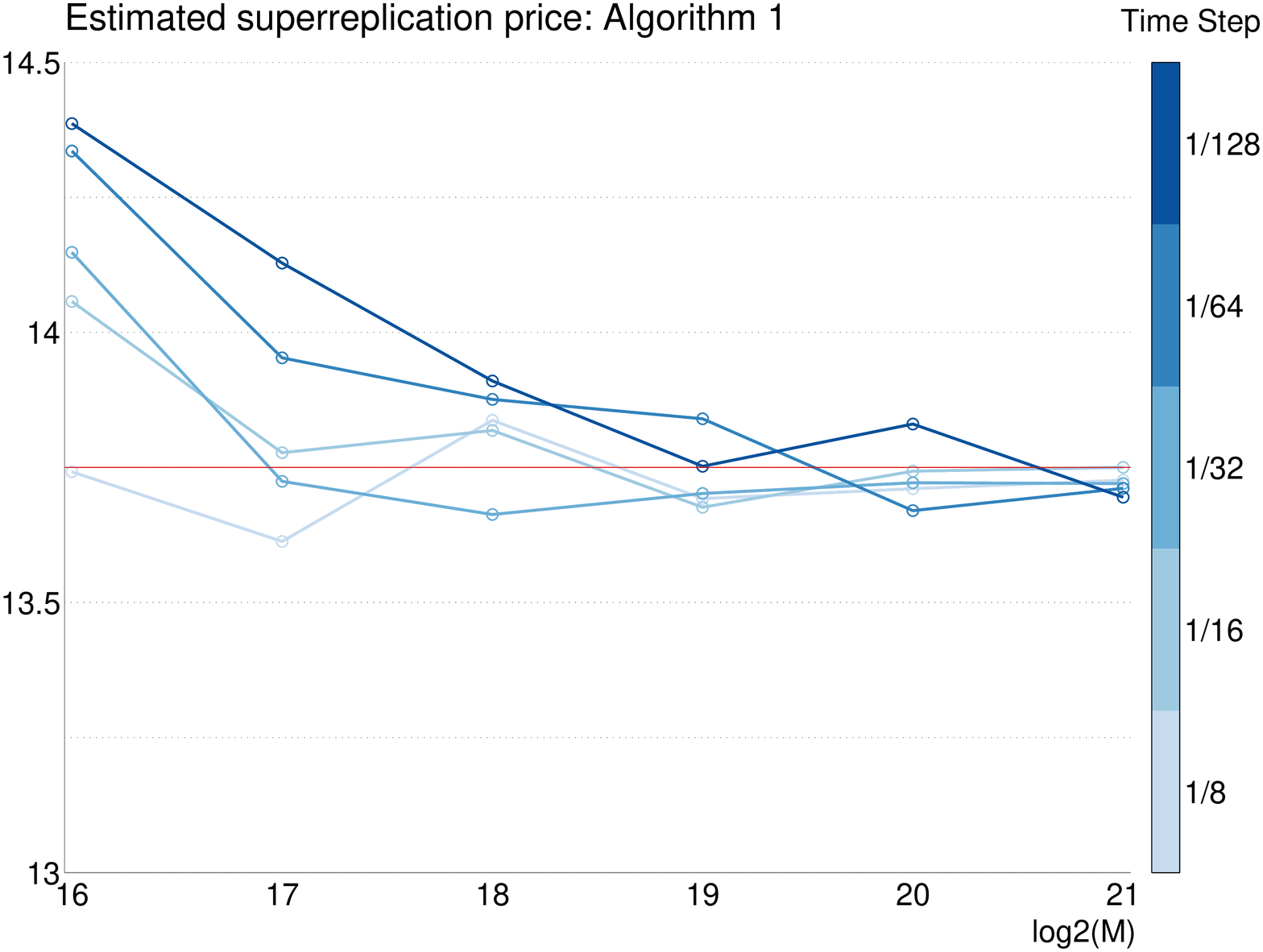}

}\subfloat[\label{fig:OutperformerCorrel2}Second Algorithm]{\includegraphics[width=0.4\paperwidth]{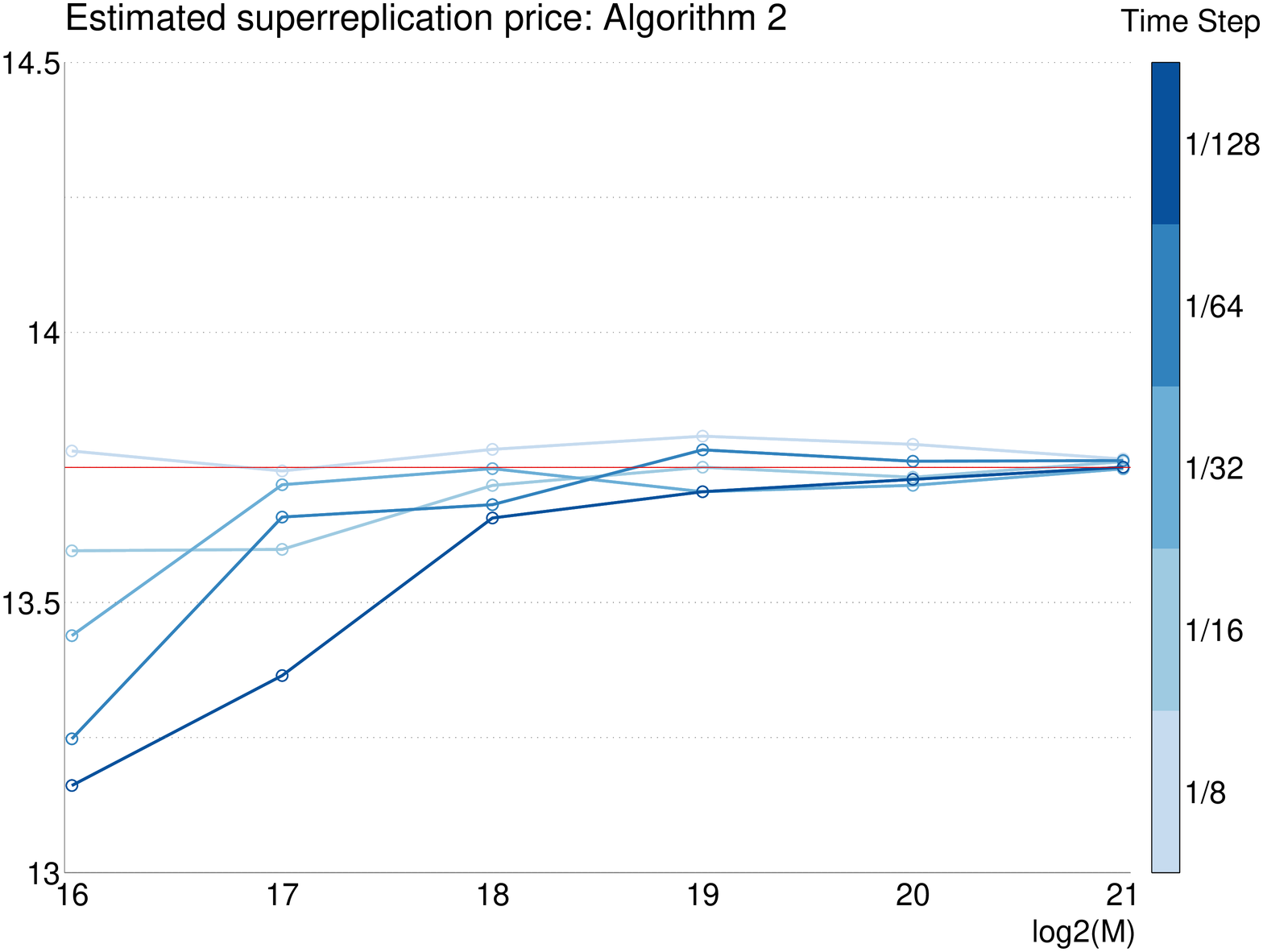}

}\caption{\label{fig:UVM-OPF-2D-RHO}Price of Outperformer Option ($\rho=-0.5$)}
\end{figure}

The same behaviour can be observed. Both algorithms are accurate here
($13.69$ ($-0.4\%$) and $13.75$ ($-0\%$)).

As the estimate from the first algorithm happens to lie below the
true price, we take advantage of this result to recall from the introduction
of this subsection that the bias of the first algorithm bears one
more source of error (the regression bias) than the bias of the second
algorithm. This means that in general the sign of the bias wrt. the
true discretized price is more reliable with the second algorithm.
With this observation in mind, we propose, from the two estimates
$P_{1}$ and $P_{2}$ computed by the two algorithms, to consider
the following general estimate $P$:
\[
P:=\max\left(P_{2},\frac{P_{1}+P_{2}}{2}\right)
\]

Indeed, if $P_{1}\geq P_{2}$ (which is the expected behaviour), then
$P:=\frac{P_{1}+P_{2}}{2}$ may provide a better estimate than both
$P_{1}$ and $P_{2}$ separately (as is the case for the call spread
example from Figure \ref{fig:UVM-CS-1D}). However, when $P_{1}<P_{2}$
(which is not expected), then, recalling that $P_{2}$ may be more
accurate than $P_{1}$, it is better to consider $P:=P_{2}$ (as is
the case here of this outperformer option with $\rho=-0.5$). In the
following, we will call $P$ the mid-estimate (with a slight abuse
of terminology, as $P$ is usually but not always the average between
$P_{1}$ and $P_{2}$).

\paragraph*{Outperformer spread option:}

We now analyze a more complex payoff. Consider an outperformer spread
option, with payoff $\left(S_{2}\left(T\right)-K_{1}S_{1}\left(T\right)\right)^{+}-\left(S_{2}\left(T\right)-K_{2}S_{1}\left(T\right)\right)^{+}$,
time horizon $T=1$ and constant correlation $\rho=-0.5$. The true
(PDE) price is $\mathcal{C}_{PDE}=11.41$, and the Black-Scholes price
with mid-volatility is $\mathcal{C}_{BS}=9.04$. We use the following
set of basis functions:
\begin{align*}
\phi\left(t,s_{1},s_{2},\sigma_{1},\sigma_{2}\right)=s_{1}\times\left(K_{2}-K_{1}\right)\times & \mathcal{S}\left(\beta_{0}+\beta_{1}\frac{s_{2}}{s_{1}}+\beta_{2}\left(\frac{s_{2}}{s_{1}}\right)^{2}+\beta_{3}\sigma_{1}+\beta_{4}\sigma_{1}\frac{s_{2}}{s_{1}}+\beta_{5}\sigma_{1}\left(\frac{s_{2}}{s_{1}}\right)^{2}\right.\\
 & \left.+\beta_{6}\sigma_{2}+\beta_{7}\sigma_{2}\frac{s_{2}}{s_{1}}+\beta_{8}\sigma_{2}\left(\frac{s_{2}}{s_{1}}\right)^{2}\right)
\end{align*}

\begin{figure}[H]
\hspace{-4mm}\subfloat[\label{fig:OutperformerSpread1}First Algorithm]{\includegraphics[width=0.4\paperwidth]{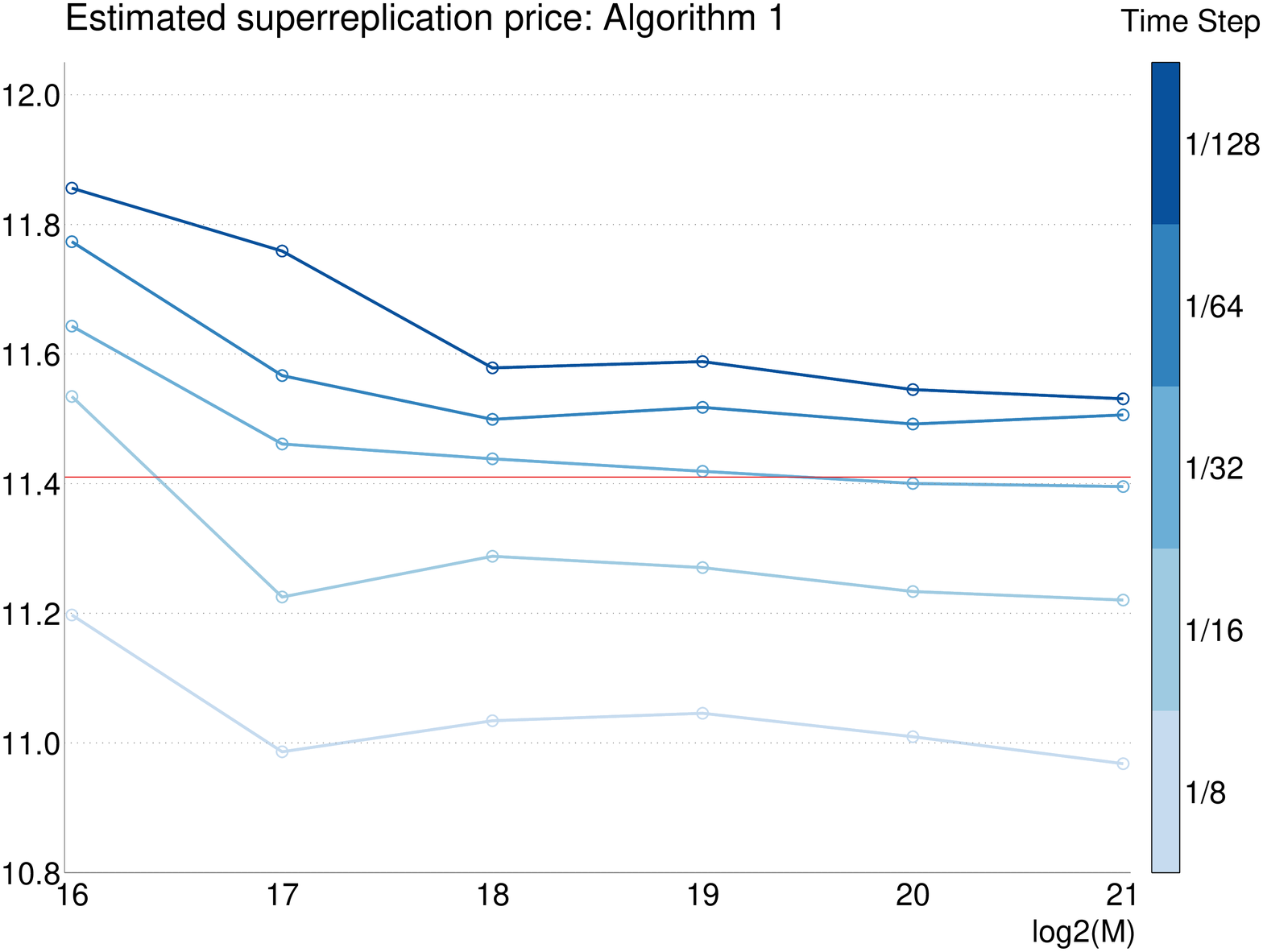}

}\subfloat[\label{fig:OutperformerSpread2}Second Algorithm]{\includegraphics[width=0.4\paperwidth]{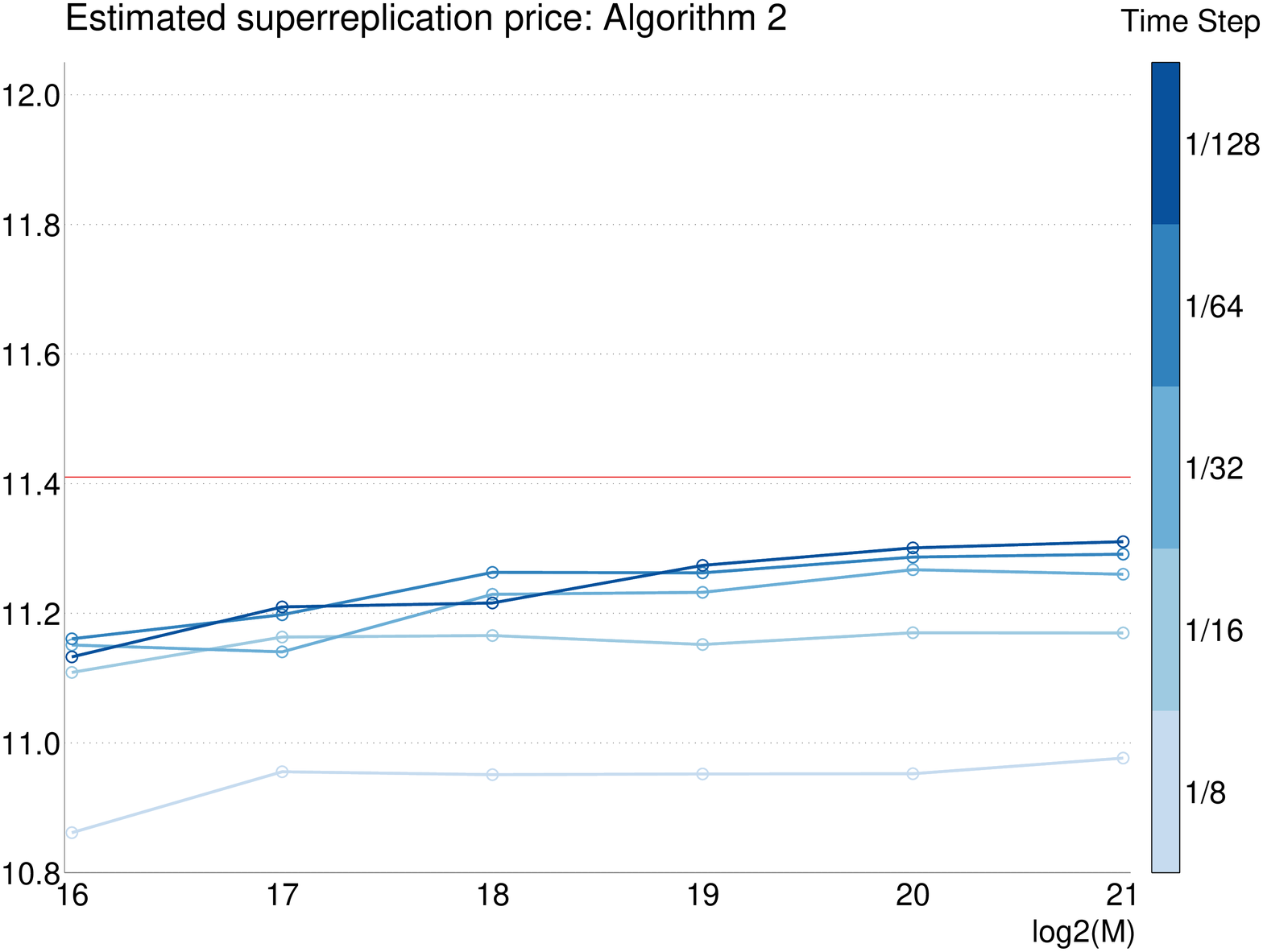}

}\caption{\label{fig:UVM-OPF-SPREAD}Price of Outperformer Spread Option ($\rho=-0.5$)}
\end{figure}

In this example, one can see that the time discretization produces
a large downward bias (as in the call spread and digital option examples),
but both algorithms behave as expected (the first algorithm produces
high estimates, the second produces low estimates, and both are close
to the true price ($11.53$ ($+1\%$) and $11.31$ ($-0.9\%$))).
Moreover, the mid-estimate $11.42$ is very accurate.

In \cite{Guyon11} is reported the estimate $10.83$ for $\left(1/\Delta_{t},\log_{2}\left(M\right)\right)=\left(8,20\right)$,
which is slightly worse than our estimates for the same choice of
$M$ and $\Delta_{t}$ ($11.01$ and $10.95$), but the difference
can be due to the different choice of basis. However, the three estimates
are well below the true price, and our numerical results indicate
that the reason is that $\Delta_{t}=1/8$ is too large a time step.

This suggests that the estimates from \cite{Guyon11} could be improved
by considering smaller time steps. However, as acknowledged in their
paper, the second-order BSDE method does not work properly when $\Delta_{t}$
is too small. Indeed, their BSDE scheme makes use of the first order
component $Z$ and the second order component $\Gamma$. The problem
here is that, for fixed $M$, the variance of the estimators of $Z$
and $\Gamma$ tends to infinity when $\Delta_{t}$ tends to zero.
However, as detailed in \cite{Alanko13}, this problem can be completely
solved by amending the estimators using appropriate variance reduction
terms. Therefore, in our opinion, a fair comparison of the jump-constrained
BSDE approach and the second-order BSDE approach would require the
use of the variance reduction method from \cite{Alanko13} to allow
for smaller time steps for the second-order BSDE approach.

As a final numerical example, we consider again the same outperformer
spread option, with the exception that the correlation $\rho$ is
now considered uncertain, within $\left[-0.5,0.5\right]$. The true
(PDE) price is $\mathcal{C}_{\mathrm{PDE}}=12.83$, and the Black-Scholes
price with mid-volatility is $\mathcal{C}_{BS}=9.24$. We use the
following basis functions:
\begin{align*}
\phi\left(t,s_{1},s_{2},\sigma_{1},\sigma_{2},\rho\right)=s_{1}\times\left(K_{2}-K_{1}\right)\times & \mathcal{S}\left(\beta_{0}+\beta_{1}\frac{s_{2}}{s_{1}}+\beta_{2}\left(\frac{s_{2}}{s_{1}}\right)^{2}+\beta_{3}\sigma_{1}+\beta_{4}\sigma_{1}\frac{s_{2}}{s_{1}}+\beta_{5}\sigma_{1}\left(\frac{s_{2}}{s_{1}}\right)^{2}\right.\\
 & \left.+\beta_{6}\sigma_{2}+\beta_{7}\sigma_{2}\frac{s_{2}}{s_{1}}+\beta_{8}\sigma_{2}\left(\frac{s_{2}}{s_{1}}\right)^{2}+\beta_{6}\rho+\beta_{7}\rho\frac{s_{2}}{s_{1}}+\beta_{8}\rho\left(\frac{s_{2}}{s_{1}}\right)^{2}\right)
\end{align*}

Remark that at each time step we perform here a five-dimensional regression.

\begin{figure}[H]
\hspace{-4mm}\subfloat[\label{fig:OutperformerSpreadCorrel1}First Algorithm]{\includegraphics[width=0.4\paperwidth]{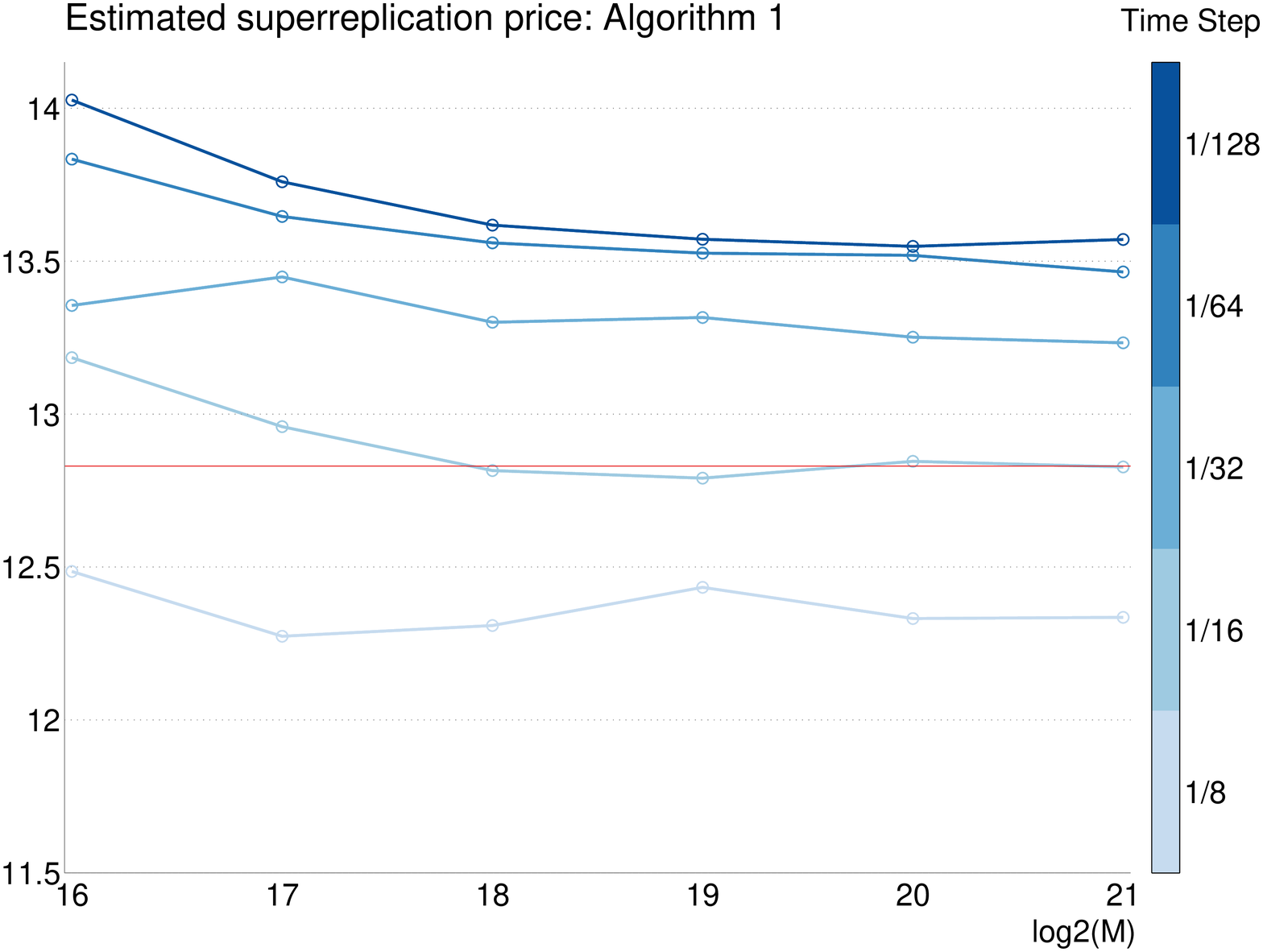}

}\subfloat[\label{fig:OutperformerSpreadCorrel2}Second Algorithm]{\includegraphics[width=0.4\paperwidth]{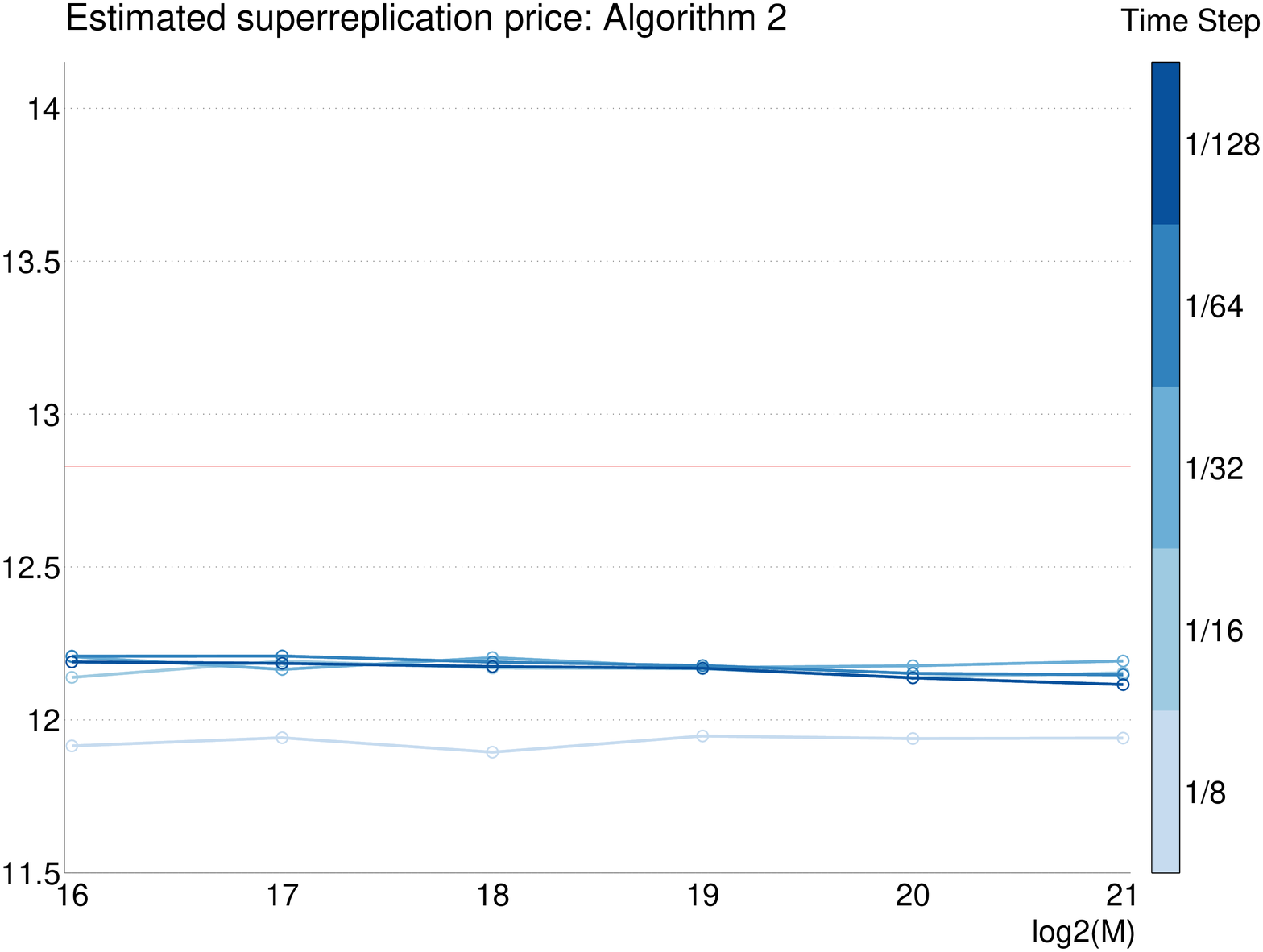}

}\caption{\label{fig:UVM-OPF-SPREAD-RHO}Price of Outperformer Spread Option
($\rho$ uncertain $\in\left[-0.5,0.5\right]$)}
\end{figure}

On this example, we observe a wide gap between the two estimates $13.57$
($+5.8\%$) and $12.12$ ($-5.6\%$) ($\left(1/\Delta_{t},\log_{2}\left(M\right)\right)=\left(128,21\right)$).
As neither the number of Monte Carlo simulation nor the discretization
time step seem able to narrow the gap, it means that it is due to
the chosen regression basis. Indeed, our basis is such that the optimal
volatilities and correlation are of bang-bang type, as in the previous
examples. However, unlike the previous examples, here both the volatilities
and the correlation are uncertain, and in this case it is known (cf.
\cite{Mrad08} for instance) that the optimum is not of a bang-bang
type. Therefore, one should look for a richer regression basis in
order to narrow the estimation gap on this specific example. Remark
however that the mid-estimated $12.84$ remains very accurate. On
the same example and with another regression basis, \cite{Guyon11}
manage to reach a price of $12.54$ for $\left(1/\Delta_{t},\log_{2}\left(M\right)\right)=\left(8,20\right)$.

To conclude these subsection, here are the differences we could notice
between the jump-constrained BSDE approach and the second-order BSDE
approach applied to the problem of pricing by simulation under uncertain
volatility model:
\begin{itemize}
\item Both are forward-backward schemes. Thus, the first step is to simulate
the forward process. At this stage the jump-constrained BSDE approach
is advantaged, because its forward process is a simple Markov process,
therefore easy to simulate. Its randomization of the control is fully
justified mathematically. On the contrary, the second-order BSDE requires
to resort to heuristics in order to simulate the forward process despite
the fact that the control is involved in its dynamics. \cite{Guyon11}
propose to use an arbitrary constant volatility (the mid-volatility)
to simulate the forward process, and they notice that the specific
choice of prior-volatility does impact substantially the resulting
estimates.
\item Then comes the estimation of the backward process. If both schemes
require to perform regressions, this step is more difficult in the
jump-constrained BSDE approach, because the dimensionality of the
regressions is higher as the state process contains the randomized
controls. In particular the choice of regression basis is more difficult.
\item On the set of options considered here and within the same range of
numerical parameters $M$ and $\Delta_{t}$ we could not detect any
significant and systematic difference between the two algorithms.
Nevertheless, we strongly suggest the following two points:
\item First, the second-order BSDE approach would strongly benefit from
the use of the variance reduction method from \cite{Alanko13}. It
would allow for smaller time steps to be considered, and therefore
allow for a sounder and more precise numerical comparison between
the two approaches. Indeed, the accurate estimates recorded in \cite{Guyon11}
for very large time steps may be, as in Figure \ref{fig:OutperformerSpreadCorrel1}
for $\Delta_{t}=1/16$, an incidental cancellation of biases of opposite
signs. The significant quantity is the level where the estimates converge
for small $\Delta_{t}$.
\item Second, to complement the downward biased, ``Longstaff-Schwartz like''
estimator considered in \cite{Guyon11}, we suggest the computation
of the upward biased, ``Tsitsiklis-van Roy like'' estimator, as
we did in this paper, as both estimators appear to be informative
in a complementary fashion, and the mid-estimator proposed here (which
requires both estimators) seems to perform staggeringly well.
\end{itemize}

\section{Conclusion\label{sec:Conclusion}}

We proposed in this paper a general probabilistic numerical algorithm,
combining Monte Carlo simulations and empirical regressions, which
is able to solve numerically very general HJB equations in high dimension. 
That includes general stochastic control problems with controlled volatility, possibly degenerate, but
more generally, it can solve any jump-constrained BSDE (\cite{Kharroubi12}).

We initiated a partial analysis of the theoretical error of the scheme,
and we provided several numerical application of the scheme on the
problem of pricing under uncertain volatility, the results of which
are very promising.

In the future, we would like to extend this work in the following
direction:
\begin{itemize}
\item First, we would like to manage to obtain a comprehensive analysis
of the error of the scheme, including the empirical regression step.
\item Then, we would like to perform a more systematic numerical comparison
with the alternative scheme described in \cite{Guyon11}, taking into
account our empirical findings.
\item Finally, we would like to extend the general methodology of control
randomization and subsequent constraint on resulting jumps to more
general problems, like HJB-Isaacs equations or even mean-fields games,
with possible advances on the numerical solution of such problems. 
\end{itemize}
\appendix
\bibliographystyle{abbrv}
\phantomsection\addcontentsline{toc}{section}{\refname}\bibliography{BiblioPHD.bib}

\end{document}